\documentclass[12pt]{amsart}
\usepackage{graphicx}
\usepackage{amsmath}
\usepackage{amssymb}
\usepackage{setspace}
\usepackage{color}
\newtheorem{thm}{Theorem}[section]

\newtheorem{lem}[thm]{Lemma}

\theoremstyle{definition}

\theoremstyle{remark}

\numberwithin{equation}{section}

\begin{document}

\title[Sharp reversed HLS inequality]{Sharp reversed Hardy-Littlewood-Sobolev inequality with extended kernel}

\author{Wei Dai, Yunyun Hu, Zhao Liu}

\address{School of Mathematical Sciences, Beihang University (BUAA), Beijing 100083, P. R. China}
\email{weidai@buaa.edu.cn}

\address{Department of Mathematical Sciences, Yeshiva University, New York, NY 10033 USA}
\email{yhu2@mail.yu.edu}

\address{School of Mathematics and Computer Science, Jiangxi Science and Technology Normal University, Nanchang 330038, P. R. China}
\email{liuzhao@mail.bnu.edu.cn}

\thanks{Corresponding author: Zhao Liu at liuzhao@mail.bnu.edu.cn. \\
Wei Dai is supported by the NNSF of China (No. 11971049) and the Fundamental Research Funds for the Central Universities, Zhao Liu is supported by the NNSF of China (No. 11801237, 11926324).}

\date{}

\begin{abstract}
In this paper, we prove the following reversed Hardy-Littlewood-Sobolev inequality with extended kernel
\begin{equation*}
\int_{\mathbb{R}_+^n}\int_{\partial\mathbb{R}^n_+} \frac{x_n^\beta}{|x-y|^{n-\alpha}}f(y)g(x) dydx\geq
C_{n,\alpha,\beta,p}\|f\|_{L^{p}(\partial\mathbb{R}_+^n)}
\|g\|_{L^{q'}(\mathbb{R}_+^n)}
\end{equation*}
for any nonnegative functions $f\in L^{p}(\partial\mathbb{R}_+^n)$ and $g\in L^{q'}(\mathbb{R}_+^n)$, where $n\geq2$, $p,\ q'\in (0,1)$, $\alpha>n$, $0\leq\beta<\frac{\alpha-n}{n-1}$, $p>\frac{n-1}{\alpha-1-(n-1)\beta}$ such that $\frac{n-1}{n}\frac{1}{p}+\frac{1}{q'}-\frac{\alpha+\beta-1}{n}=1$. We prove the existence of extremal functions for the above inequality. Moreover, in the conformal invariant case, we classify all the extremal functions and hence derive the best constant via a variant method of moving spheres, which can be carried out \emph{without lifting the regularity of Lebesgue measurable solutions}. Finally, we derive the sufficient and necessary conditions for existence of positive solutions to the Euler-Lagrange equations by using Pohozaev identities. Our results are inspired by Hang, Wang and Yan \cite{HWY}, Dou, Guo and Zhu \cite{DGZ} for $\alpha<n$ and $\beta=1$, and Gluck \cite{Gl} for $\alpha<n$ and $\beta\geq0$.
\end{abstract}

\maketitle {\small {\bf Keywords:} Existence of extremal functions; Euler-Lagrange equations; Pohozaev identity; Hardy-Littlewood-Sobolev inequality.\\

}

\section{Introduction}
The classical Hardy-Littlewood-Sobolev inequality that was obtained by Hardy and Littlewood \cite{Hardy}
for $n = 1$ and by Sobolev \cite{Sobolev} for general $n$ states that
\begin{equation}\label{HL1}
\int_{\mathbb{R}^n}\int_{\mathbb{R}^n}|x-y|^{-(n-\alpha)} f(x)g(y)dxdy\leq C_{\alpha,n,p}\|f\|_{L^{p}(\mathbb{R}^n)}\|g\|_{L^{q'}(\mathbb{R}^n)}
\end{equation}
with $1<p,q'<\infty, 0<\alpha<n$ and $\frac{1}{p}+\frac{1}{q'}+\frac{n-\alpha}{n}=2$.

\medskip

This inequality has become an important tool in the analysis of PDEs. To find the best constant, the classification of the extremal functions plays a crucial role. When $p=q'=\frac{2n}{n+\alpha}$, the Euler-Lagrange equation of the extremals to the Hardy-Littlewood-Sobolev inequality in Euclidean space is a conformal invariant integral equation. In \cite{Lieb}, Lieb proved the existence of extremal functions for the inequality \eqref{HL1} by employing rearrangement inequalities. In addition, he classified extremal functions and computed the sharp constant when $p$ or $q'$ is equal to $2$, or $p=q'=\frac{2n}{n+\alpha}$.

Inequality \eqref{HL1} implies many important geometrical inequalities, such as the Gross logarithmic Sobolev inequality \cite{Gr}, and the Moser-Onofri-Beckner inequality \cite{Beckner1}. It is also well-known that if one picks $\alpha=2$ and $p=q'=\frac{2n}{n+2}$, then the Hardy-Littlewood-Sobolev inequality is in fact equivalent to the Sobolev inequality by using Green's representation formula. By using the competing symmetry method, Carlen and Loss \cite{CarL} provided a different proof from Lieb's of the sharp constants and extremal functions in the diagonal cases $p=q'=\frac{2n}{n+\alpha}$. By using the reflection positivity of inversions in spheres, Frank and Lieb \cite{FL1} also offered a new proof in some special diagonal cases. In \cite{FL2}, Frank and Lieb further employed a rearrangement-free technique developed in \cite{FL3} to recapture the best constant of inequality \eqref{HL1}.

Folland and Stein \cite{FS} extended the inequality \eqref{HL1} to the Heisenberg group and established the Hardy-Littlewood-Sobolev inequality on Heisenberg group. Frank and Lieb \cite{FL3} classified the extremals of the Hardy-Littlewood-Sobolev inequality on Heisenberg group in the diagonal cases. This extends the earlier work of Jerison and Lee \cite{JL} for sharp constants and extremals for the Sobolev inequality on the Heisenberg group in the conformal case in their study of CR Yamabe problem. Furthermore, Han, Lu and Zhu \cite{HLZ} established the double weighted Hardy-Littlewood-Sobolev inequality (namely, Stein-Weiss inequality) on the Heisenberg group and investigated the regularity and asymptotic behavior of the extremal functions. Recently, Chen, Lu and Tao \cite{CLTao} used the concentration-compactness principle to obtain existence of extremals of the Stein-Weiss inequality on the Heisenberg group for all indices.

Inequality \eqref{HL1} and its extensions have many applications in partial differential equations. Some remarkable extensions have already been obtained on the upper half space by Dou and Zhu \cite{DZ1} and on compact Riemannian manifolds by Han and Zhu \cite{Han}. The reversed (weighted) Hardy-Littlewood-Sobolev inequalities were derived in \cite{CLLT1,DZ2,Ngo1,Ngo2}. For more results concerning the (weighted) Hardy-Littlewood-Sobolev inequality and Hardy-Sobolev equations, please refer to \cite{Beckner2,Brascamp,CD,Chen,CL1,CL2,CJLL,CLZ1,CLZ2,DFHQW,DL,DL2,DLQ,LL,Liu,L,Lu,SW} and the references therein.

\medskip

Closely related to the Hardy-Littlewood-Sobolev inequality, Gluck \cite{Gl} recently established the following half-space version of sharp Hardy-Littlewood-Sobolev inequality involving general kernels (extended kernels) in the conformal invariant case $p=\frac{2(n-1)}{n+\alpha-2}$ and $q'=\frac{2n}{n+\alpha+2\beta}$:
\begin{equation}\label{equality00}
\Big|\int_{\mathbb{R}_+^n}\int_{\partial\mathbb{R}^n_+} K_{\alpha,\beta}(x'-y',x_n)f(y)g(x) dydx\Big|\leq C_{n,\alpha,\beta}
\|f\|_{L^p(\partial\mathbb{R}_+^n)}\|g\|_{L^{q'}(\mathbb{R}_+^n)},
\end{equation}
where the kernel
$$K_{\alpha,\beta}(x',x_n):=\frac{x_n^\beta}{(|x'|^2+x_n^2)^{\frac{n-\alpha}{2}}}, \quad  x=(x',x_n)\in \mathbb{R}^{n-1}\times (0,\infty)$$
with $\beta\geq0$, $0<\alpha+\beta<n-\beta$,
\begin{equation}\label{tiaojian}
\frac{n-\alpha-2\beta}{2n}+\frac{n-\alpha}{2(n-1)}<1,
\end{equation}
and
$$\mathbb{R}^{n}_{+}=\{x=(x_1,\cdots,x_n)\in \mathbb{R}^n \mid x_n>0\}.$$

In the special case $\alpha=0$ and $\beta=1$, $K_{\alpha,\beta}$ becomes the classical Poisson kernel, for which Hang, Wang and Yan \cite{HWY} derived the Hardy-Littlewood-Sobolev inequality and obtained the existence of extremals through a concentration-compactness principle. For the conformal invariant case, they classified the extremal functions of the inequality, and thus computed the sharp constant.

For $\alpha\in (0,1)$, $\beta=1-\alpha$, the kernel $K_{\alpha,\beta}$ is related to the divergence form operator $u\mapsto div(x_n^\alpha \nabla u)$
(the harmonic extension operator) on the half space. Chen \cite{Chens} established the sharp Hardy-Littlewood-Sobolev inequality \eqref{equality00} in such cases. He also generalized Carleman's inequality for harmonic functions in the plane to poly-harmonic functions in higher dimensions.

When $\beta=0$, Dou and Zhu \cite{DZ1} studied the sharp Hardy-Littlewood-Sobolev inequality on upper half spaces and obtained the existence of extremal functions.

When $\beta=1$, Dou, Guo and Zhu \cite{DGZ} investigated the integral inequality \eqref{equality00} for some special indices through the methods based on conformal transformations. Subsequently, Chen, Lu and Tao \cite{CLT} derived the Hardy-Littlewood-Sobolev inequality for all critical indices when $$\frac{n-1}{n}\frac{1}{p}+\frac{1}{q}-\frac{\alpha+\beta-1}{n}=1.$$
Furthermore, Chen, Liu, Lu and Tao \cite{CLLT2} extended it to the weighted Hardy-Littlewood-Sobolev inequality.

In the case $\alpha+\beta>1$, Liu \cite{Liu2} generalized the Hardy-Littlewood-Sobolev inequality with general kernel in the conformal invariant case for all critical indices.

\medskip

In this paper, we prove the reversed Hardy-Littlewood-Sobolev inequality with extended kernel in the half space.

\begin{thm}\label{theorem1}
Let $n\geq2$, $p, q'\in (0,1)$, $\alpha>n$, $0\leq\beta<\frac{\alpha-n}{n-1}$, $p>\frac{n-1}{\alpha-1-(n-1)\beta}$ and suppose that $\alpha$, $\beta$, $p$, $q'$ satisfy
$$\frac{n-1}{n}\frac{1}{p}+\frac{1}{q'}-\frac{\alpha+\beta-1}{n}=1.$$
Then, there is a constant $C_{n,\alpha,\beta,p}>0$ such that for any nonnegative functions $f\in L^p(\partial\mathbb{R}_+^n)$ and $g\in L^{q'}(\mathbb{R}_+^n)$,
\begin{equation}\label{equality}
\int_{\mathbb{R}_+^n}\int_{\partial\mathbb{R}^n_+} \frac{x_n^\beta }{|x-y|^{n-\alpha}} f(y)g(x) dydx\geq
C_{n,\alpha,\beta,p}\|f\|_{L^p(\partial\mathbb{R}_+^n)}\|g\|_{L^{q'}(\mathbb{R}_+^n)}.
\end{equation}
\end{thm}

\medskip

Define
\begin{equation}\label{operator}
  Tf(x):=\int_{\partial\mathbb{R}^n_+}\frac{x_n^\beta}{|x-y|^{n-\alpha}}f(y)dy.
\end{equation}

Throughout this paper, we always assume that $q$ and $q'$ are conjugate numbers. That is, $q$ and $q'$ satisfy $\frac{1}{q}+\frac{1}{q'}=1$. By duality, it is easy to verify that the inequality \eqref{equality} is equivalent to the following inequality:
\begin{equation}\label{dengjiabudengshi}
\|Tf\|_{L^{q}(\mathbb{R}^n_+)}\geq C_{n,\alpha,\beta,p}\|f\|_{L^{p}(\partial\mathbb{R}^n_+)}.
\end{equation}

\medskip

Once we have established the revered Hardy-Littlewood-Sobolev inequality with extended kernel, it is natural to ask whether the extremal functions for inequality \eqref{equality} exist or not. To answer this question, we turn to consider the following minimizing problem:
\begin{equation}\label{max}
C^{*}_{n,\alpha,\beta,p}:=\inf\left\{\|Tf\|_{L^q(\mathbb{R}^n_+)} \,\big|\, \|f\|_{L^p(\partial\mathbb{R}^n_+)}=1, \, f\geq 0\right\},
\end{equation}
where $p$ and $q$ satisfy
$$\frac{1}{q}=\frac{n-1}{n}\left(\frac{1}{p}-\frac{\alpha+\beta-1}{n-1}\right), \quad \beta q+1>0.$$

\medskip

It is easy to see that the extremals of inequality \eqref{dengjiabudengshi} solve the minimizing problem \eqref{max}. We will use the rearrangement inequality to prove the attainability of minimizers for minimizing problem \eqref{max}.

\begin{thm}\label{theorem2}
There exists a function $f\in L^p(\partial \mathbb{R}^n_+)$ satisfying $f\geq0$ and $\|f\|_{L^p(\partial \mathbb{R}^n_+)}=1$, such that $\|Tf\|_{L^q( \mathbb{R}^n_+)}=C^*_{n,\alpha,\beta,p}$.
\end{thm}

The Euler-Lagrange equation for extremal functions to inequality \eqref{dengjiabudengshi}, up to a constant multiplier, is given by
\begin{equation}\label{Eulereq}
f^{p-1}(y)=\int_{\mathbb{R}^n_+}\frac{x_n^\beta}{|x-y|^{n-\alpha}}\left(Tf(x)\right)^{q-1}dx.
\end{equation}

Let
$$u(y):=f^{p-1}(y), \quad v(x):=x_{n}^{-\beta}Tf(x).$$
Denote
$$-\theta:=\frac{1}{p-1}, \quad -k:=q-1.$$
The Euler-Lagrange equation \eqref{Eulereq} can be rewritten as the following integral system:
\begin{equation}\label{integral}\begin{cases}
u(y)=\int_{\mathbb{R}^n_+}\frac{x_n^{\beta(1-k)}}{|x-y|^{n-\alpha}}v^{-k}(x)dx, \quad y\in\partial\mathbb{R}^n_+,\\ \\
v(x)=\int_{\partial\mathbb{R}^n_+}\frac{1}{|x-y|^{n-\alpha}}u^{-\theta}(y) dy, \quad x\in\mathbb{R}^n_+,
\end{cases}\end{equation}
where $\frac{1}{k-1}=\frac{n-1}{n}\left(\frac{\alpha+\beta-n}{n-1}-\frac{1}{\theta-1}\right)$ with $\beta(1-k)+1>0$.

\medskip

The classification of solutions to integral system \eqref{integral} plays a key role in deriving the sharp constant for Hardy-Littlewood-Sobolev inequality \eqref{equality}.

In the case $\beta=0$, Li \cite{L} first classified positive Lebesgue measurable solutions to single integral equation with negative exponents in the whole space. The classification of positive Lebesgue measurable solutions to integral system \eqref{integral} were obtained by Dou and Zhu \cite{DZ2} in the whole space and by Ng\^{o} and Nguyen \cite{Ngo2} in the half space.

\medskip

In all the above papers \cite{L,DZ2,Ngo2}, the authors can improve the regularity of positive Lebesgue measurable solutions to smooth solutions via the standard bootstrap method. In order to start moving the spheres, the regularity should be improved to \emph{(at least) $C^1$}.

Nevertheless, due to the presence of the factor $x_n^\beta$ in the numerator of extended kernels, it is very difficult for us to apply the standard bootstrap method to lift the regularity of Lebesgue measurable solutions so as to carry out the method of moving spheres.

\medskip

In this paper, by applying some new ideas, we obtain the classification of positive Lebesgue measurable solutions for system \eqref{integral} \emph{without lifting the regularity of Lebesgue measurable solutions}. Being different from Li in \cite{L}, we dilate the spheres $\emph{S}_\lambda$ from the neighborhood of singular points to the limiting radius $\bar{\lambda}$. By exploiting the \emph{Spherically Narrow Region Maximum Principle} in integral forms, we use a slight variant of the method of moving spheres to classify all positive Lebesgue measurable solutions to integral system \eqref{integral}. We believe that our methods can be applied to many other integral equations with various kernels.

\medskip

We prove
\begin{thm}\label{theoremfen}
Let $(u,v)$ be a pair of positive Lebesgue measurable solutions on $\partial\mathbb{R}^n_{+}\times \mathbb{R}^n_{+}$ satisfying \eqref{integral}. Assume that
$$k\leq\frac{n+\alpha+2\beta}{\alpha+2\beta-n}, \qquad \theta\leq\frac{n+\alpha-2}{\alpha-n}.$$
Then, $u$ and $v$ must be the form of
\begin{equation*}
u(y)=c_1\left(\frac{d}{1+d^{2}|y-z_0|^2}\right)^{\frac{n-\alpha}{2}}, \qquad v(y,0)=c_2\left(\frac{d}{1+d^{2}|y-z_0|^2}\right)^{\frac{n-\alpha}{2}}, \qquad \forall \, y\in \partial\mathbb{R}^{n}_+
\end{equation*}
for some $z_0\in\partial \mathbb{R}^n_{+}$, $c_1>0$, $c_2>0$ and $d>0$. Furthermore, we have
$$k=\frac{n+\alpha+2\beta}{\alpha+2\beta-n},\ \ \ \theta=\frac{n+\alpha-2}{\alpha-n}.$$
Moreover, in such conformally invariant case, the value of the best constant in \eqref{dengjiabudengshi} is
\begin{equation}\label{sharp}
C^*_{n,\alpha,\beta,p}=\left(n \nu_n\right)^{-\frac{n+\alpha-2}{2(n-1)}}
\left(\int_{ B_1}\left|\int_{\partial B_1}\left(\frac{1-|\xi|^2}{2}\right)^\beta|\eta-\xi|^{\alpha-n}
d\eta\right|^{\frac{2n}{n-\alpha-2\beta}}d\xi\right)^{\frac{n-\alpha-2\beta}{2n}},
\end{equation}
where $\nu_{n}$ denotes the volume of the unit ball in $\mathbb{R}^{n}$.
\end{thm}

The method of moving spheres can be applied to capture the explicit form of solutions directly rather than going through the procedure of deriving radial symmetry of solutions and then classifying radial solutions. For more results related to the method of moving spheres or moving planes, please see \cite{CD,Chen,CGS,CL0,CL2,DFHQW,DL,DL2,DLL,DLQ,DQ,GNN,Liu,LD,L,LZH,LZ,Lu,Pa} and the references therein.

\medskip

As a consequence of Theorem \ref{theoremfen}, we can derive the following sharp reversed Hardy-Littlewood-Sobolev inequality on $\mathbb{R}^{n}_{+}$ in the conformally invariant cases $p=\frac{2(n-1)}{n+\alpha-2}$ and $q=\frac{2n}{n-\alpha-2\beta}$.
\begin{thm}\label{SR-HLS}
Assume $n\geq2$, $\alpha>n$ and $0\leq\beta<\frac{\alpha-n}{2(n-1)}$. For all nonnegative function $f\in L^{\frac{2(n-1)}{n+\alpha-2}}(\partial\mathbb{R}^{n}_{+})$, we have
\begin{equation}\label{SRHLS}
  \|Tf\|_{L^{\frac{2n}{n-\alpha-2\beta}}(\mathbb{R}^n_+)}\geq C^{\ast}_{n,\alpha,\beta,\frac{2(n-1)}{n+\alpha-2}}\|f\|_{L^{\frac{2(n-1)}{n+\alpha-2}}(\partial\mathbb{R}^n_+)},
\end{equation}
where $Tf$ is defined in \eqref{operator} and the sharp constant $C^{\ast}_{n,\alpha,\beta,\frac{2(n-1)}{n+\alpha-2}}$ is given by \eqref{sharp}. The equality holds if and only if
\begin{equation}\label{extremal}
  f(y)=c\left(\frac{d}{1+d^{2}|y-z_0|^2}\right)^{\frac{n+\alpha-2}{2}}, \qquad \forall \, y\in\partial\mathbb{R}^{n}_{+}
\end{equation}
for some $z_0\in\partial \mathbb{R}^n_{+}$, $c>0$ and $d>0$.
\end{thm}

\medskip

Obviously, for any extremal function $f$ to inequality \eqref{dengjiabudengshi}, $(f^{p-1},x_{n}^{-\beta}Tf)$ solves the integral system \eqref{integral} up to a constant multiplier. In light of Theorem \ref{theorem2}, we deduce that the sufficient condition for existence of positive solutions $(u,v)$ satisfying $(u,x_{n}^{\beta}v)\in L^{1-\theta}(\partial\mathbb{R}^{n}_{+})\times L^{1-k}(\mathbb{R}^{n}_{+})$ to the integral system \eqref{integral} is
\begin{equation}\label{SC}
  \frac{n-1}{\theta-1}+\frac{n}{k-1}=\alpha+\beta-n \qquad \text{and} \qquad \beta(1-k)+1>0.
\end{equation}

We will use the Pohozaev identities to prove that the sufficient condition \eqref{SC} is also a necessary condition for the existence of $C^{1}$ positive solutions to the system \eqref{integral}.

\begin{thm}\label{theorem4}
For $n\geq2$, $\alpha>n$, $\beta\geq0$, $\theta>0$, $k>0$, $\theta\neq1$, $k\neq1$ satisfying $\beta(1-k)+1>0$, assume that the system \eqref{integral} admits a pair of positive $C^1$ solutions $(u,v)$, then a necessary condition for $\theta$ and $k$ is
$$\frac{n-1}{\theta-1}+\frac{n}{k-1}=\alpha+\beta-n.$$
\end{thm}

\medskip

The rest of this paper are organized as follows. Section 2 is devoted to proving the reversed Hardy-Litttlewood-Sobolev inequality with extended kernels. In section 3, by using the rearrangement inequality, we obtain the existence of extremals to the reversed Hardy-Litttlewood-Sobolev inequality. In section 4, we use a variant of the method of moving spheres to classify all extremal functions, and compute the best constant in the conformal invariant case. In section 5, by using the Pohozaev identity in integral forms, we give a necessary condition for the existence of positive $C^{1}$ solutions to integral system \eqref{integral}.

\section{The proof of Theorem \ref{theorem1}}

In this section, we use the reversed Marcinkiewicz interpolation theorem and weak type estimate to establish the reversed Hardy-Littlewood-Sobolev inequality with the extended kernels.

\medskip

\begin{thm}\label{theorem1sub}
Let $n\geq2$, $p, q'\in(0,1)$, $\alpha>n$, $0\leq\beta<\frac{\alpha-n}{n-1}$, $p>\frac{n-1}{\alpha-1-(n-1)\beta}$ and suppose that $\alpha$, $\beta$, $p$, $q'$ satisfy
$$\frac{n-1}{n}\frac{1}{p}+\frac{1}{q'}-\frac{\alpha+\beta-1}{n}=1.$$
Then, there is a constant $C_{n,\alpha,\beta,p}>0$ such that for any nonnegative functions $f\in L^p(\partial\mathbb{R}_+^n)$ and $g\in L^{q'}(\mathbb{R}_+^n)$,
\begin{equation}\label{equalitysub}
\int_{\mathbb{R}_+^n}\int_{\partial\mathbb{R}^n_+} \frac{x_n^\beta }{|x-y|^{n-\alpha}} f(y)g(x) dydx\geq
C_{n,\alpha,\beta,p}\|f\|_{L^p(\partial\mathbb{R}_{+}^{n})}\|g\|_{L^{q'}(\mathbb{R}_{+}^{n})}.
\end{equation}
\end{thm}

\medskip

\begin{proof}
For $t>0$ and $x'\in \mathbb{R}^{n-1}$, define
$$K_t(x'):=\frac{t^\beta}{(|x'|^2+t^2)^{\frac{n-\alpha}{2}}}.$$
Then, for $x=(x',x_n)\in\mathbb{R}^n_+$, $y=(y',0)\in \partial \mathbb{R}^n_+$, we have
$$K_{\alpha,\beta}(x'-y',x_n)=\frac{x_n^\beta }{|x-y|^{n-\alpha}}=K_{x_n}(x'-y'), \qquad Tf(x)=(K_{x_n}\ast f)(x').$$

We are to prove Theorem \ref{theorem1sub} via showing inequality \eqref{dengjiabudengshi}. Let $p\in\left(\frac{n-1}{\alpha-1-(n-1)\beta},1\right)$, and $q$ satisfy $\frac{1}{q}=\frac{n-1}{n}\left(\frac{1}{p}-\frac{\alpha+\beta-1}{n-1}\right)$. By the reversed Marcinkiewicz interpolation theorem (see Proposition 2.5 in Dou and Zhu \cite{DZ2}), we only need to prove the following weak-type estimate:
\begin{equation}\label{a0}
\|Tf\|_{L_{w}^{q}(\mathbb{R}^n_+)}\geq C_{n,\alpha,\beta,p}\|f\|_{L^{p}(\partial\mathbb{R}^n_+)}, \qquad \forall \,\, f\in L^{p}(\partial\mathbb{R}^n_+), \,\, f\geq0.
\end{equation}
That is, we need to show that there is a constant $C_{n,\alpha,\beta,p}>0$ such that
$$\inf_{\lambda>0}\lambda\left|\{x\in\mathbb{R}^n_+ \mid |Tf(x)|<\lambda\}\right|^{\frac{1}{q}}\geq C_{n,\alpha,\beta,p}\|f\|_{L^{p}(\partial\mathbb{R}^n_+)}, \quad \forall f\in L^{p}(\partial\mathbb{R}^n_+), \,\, f\geq0.$$

Without loss of generality, we may assume that $\|f\|_{L^p(\partial\mathbb{R}^n_+)}=1$. Assume that $r$, $s<0$ satisfy
\begin{equation}\label{a1}
r\in \left(\frac{np}{(1-\alpha-\beta)p+n-1},0\right), \quad s<\frac{n-1}{n-\alpha}, \qquad \frac{1}{r}+1=\frac{1}{p}+\frac{1}{s}.
\end{equation}
It follows from the reversed Young equality (Lemma 2.2 in \cite{DZ2}) that, for any $a>0$,
\begin{equation*}\begin{split}
\int_{\substack{x\in\mathbb{R}^n_+ \\ 0<x_n<a}}\left|Tf(x)\right|^{r}dx
&=\int_{0}^{a}\int_{\mathbb{R}^{n-1}}\left|(K_{x_n}\ast f)(x')\right|^{r}dx'dx_n\\
&\leq\|f\|^r_{L^{p}(\mathbb{R}^{n-1})}\int_{0}^{a}\left\|K_{x_n}\right\|^r_{L^s(\mathbb{R}^{n-1})}dx_n\\
&=\int_{0}^{a}\left(\int_{\mathbb{R}^{n-1}}\frac{x_n^{\beta s}}{\left(|x'|^2+x_n^2\right)^{\frac{(n-\alpha)s}{2}}}dx'\right)^{\frac{r}{s}}dx_n\\
&=\int_{0}^{a}x_n^{\frac{(n-1)r}{s}+(\alpha+\beta-n)r}dx_n\cdot\left(\int_{\mathbb{R}^{n-1}}\frac{1}{\left(|x'|^2+1\right)^{\frac{(n-\alpha)s}{2}}}dx'\right)^{\frac{r}{s}}.
\end{split}\end{equation*}
One can deduce from \eqref{a1} that
$$\frac{(n-1)r}{s}+(\alpha+\beta-n)r>-1.$$
Then, we have
\begin{equation}\label{eq0}
\int_{\substack{x\in\mathbb{R}^n_+ \\ 0<x_n<a}}|Tf(x)|^{r}dx\leq C_1 a^{\frac{(n-1)r}{s}+(\alpha+\beta-n)r+1}.
\end{equation}

In view of the reversed H\"{o}lder inequality (Lemma 2.1 in \cite{DZ2}), we can deduce that, for any $x_{n}>0$ and $x'\in \mathbb{R}^{n-1}$,
\begin{eqnarray*}
 && (K_{x_n}\ast f)(x')\geq x_{n}^{\beta}\left(\int_{\partial\mathbb{R}^{n}_{+}}\frac{1}{\left(|x'-y'|^{2}+x_{n}^{2}\right)^{\frac{(n-\alpha)p'}{2}}}dy\right)^{\frac{1}{p'}} \\
 && \qquad\qquad\qquad \geq x_n^{\frac{n-1}{p'}+(\alpha+\beta-n)}
  \left(\int_{\mathbb{R}^{n-1}}\frac{1}{\left(1+|y'|^{2}\right)^{\frac{(n-\alpha)p'}{2}}}dy'\right)^{\frac{1}{p'}}.
\end{eqnarray*}
Consequently, we arrive at, for any $x_{n}>0$,
\begin{equation}\label{eq1}
  \inf_{x'\in{\mathbb{R}^{n-1}}}(K_{x_n}*f)(x')\geq C_{n,\alpha,p}x_n^{\frac{n-1}{p'}+(\alpha+\beta-n)}.
\end{equation}
Since $p\in\left(\frac{n-1}{\alpha-1-(n-1)\beta},1\right)$, we get that $\frac{n-1}{p'}+(\alpha+\beta-n)>0$. Thus we can derive from \eqref{eq0} and \eqref{eq1} that, for any $\lambda>0$,
\begin{equation*}\begin{split}
&\quad \left|\left\{x\in\mathbb{R}^n_+ \mid |Tf(x)|<\lambda\right\}\right| \\
&=\Big|\Big\{x\in\mathbb{R}^n_+ \,\Big|\, 0<x_n<C_{n,\alpha,\beta,p}\lambda^{\frac{p'}{n-1+p'(\alpha+\beta-n)}},\ |Tf(x)|<\lambda\Big\}\Big|\\
&\leq \frac{1}{\lambda^r}\int_{\left\{x\in\mathbb{R}^n_+ \,\big|\,0<x_n<C_{n,\alpha,\beta,p}\lambda^{\frac{p'}{n-1+p'(\alpha+\beta-n)}}\right\}}|Tf(x)|^rdx\\
&\leq C_{n,\alpha,\beta,p}\lambda^{\frac{np}{(\alpha+\beta-1)p-n+1}}\\
&= C_{n,\alpha,\beta,p}\lambda^{-q},
\end{split}\end{equation*}
which yields immediately that
\begin{equation}\label{a3}
\|Tf\|_{L_w^q(\mathbb{R}^n_+)}\geq C_{n,\alpha,\beta,p}\|f\|_{L^{p}(\partial\mathbb{R}^n_+)}.
\end{equation}
Note that inequality \eqref{a3} implies, via the reversed Marcinkiewicz interpolation theorem (Proposition 2.5 in \cite{DZ2}), that
\begin{equation*}
\|Tf\|_{L^q(\mathbb{R}^n_+)}\geq C_{n,\alpha,\beta,p}\|f\|_{L^{p}(\partial\mathbb{R}^n_+)}.
\end{equation*}
This concludes our proof of Theorem \ref{theorem1}.
\end{proof}

\section{The proof of Theorem \ref{theorem2}}

In this section, we will employ rearrangement inequality to investigate the existence of minimizers for the minimizing problem£º
\begin{equation}\label{minimum}
C^*_{n,\alpha,\beta,p}:=\inf\left\{\left\|Tf\right\|_{L^q(\mathbb{R}^n_+)} \,\big|\, \|f\|_{L^p(\partial\mathbb{R}^n_+)}=1, \, f\geq 0\right\}.
\end{equation}

We prove
\begin{thm}\label{theorem2sub}
There exists a function $f\in L^p(\partial \mathbb{R}^n_+)$ such that $f\geq0$, $\|f\|_{L^p(\partial \mathbb{R}^n_+)}=1$, and $\|Tf\|_{L^q(\mathbb{R}^n_+)}=C^{*}_{n,\alpha,\beta,p}$.
\end{thm}
\begin{proof}
Using symmetrization argument, we first show that the minimizing problem \eqref{minimum} can be attained by radially symmetric functions.

Let $u\geq0$ be a measurable function on $\mathbb{R}^n$, the symmetric rearrangement of $u$ is the nonnegative lower semi-continuous radial decreasing function $u^*$ that has the same distribution as $u$ (see e.g. \cite{LL}). For function $v>0$, define
$$v_{*}:=\left((v^{-1})^*\right)^{-1}.$$
Then, $v_*$ is radially symmetric and increasing rearrangement function for $v$. For functions $u,w\geq0$ and $v>0$, one can derive that (see Proof of Proposition 9 in Brascamp and Lieb \cite{Brascamp})
\begin{equation*}
\int_{\mathbb{R}^n}\int_{\mathbb{R}^n}u(x-y)v(x)w(y)dxdy\geq \int_{\mathbb{R}^n}\int_{\mathbb{R}^n}u^*(x-y)v_*(x)w^{*}(y)dxdy.
\end{equation*}
Suppose that $w\geq0$ and $\|w\|_{L^{q'}(\mathbb{R}^n)}=\|w^*\|_{L^{q'}(\mathbb{R}^n)}=1$ for $0<q'<1$. Then, for $q<0$ and $q'=\frac{q}{q-1}$, we can deduce that
\begin{equation}\label{eq2} \begin{split}
&\quad\,\, \|u\ast v\|_{L^{q}(\mathbb{R}^n)} \\
&=\inf_{w\geq0,\, \|w\|_{L^{q'}(\mathbb{R}^n)}=1}\int_{\mathbb{R}^n}\int_{\mathbb{R}^n}u(x)\bar{v}(x-y)w(-y)dxdy\\
&=\inf_{w\geq0,\, \|w\|_{L^{q'}(\mathbb{R}^n)}=1}\int_{\mathbb{R}^n}\int_{\mathbb{R}^n}u(x-y)\bar{v}(x)w(y)dxdy\\
&\geq\inf_{w\geq0,\, \|w\|_{L^{q'}(\mathbb{R}^n)}=1} \int_{\mathbb{R}^n}\int_{\mathbb{R}^n}u^*(x-y)\bar{v}_*(x)w^*(y)dxdy\\
&\geq\inf_{w\geq0,\, \|w\|_{L^{q'}(\mathbb{R}^n)}=1} \left(\int_{\mathbb{R}^n}
\left(\int_{\mathbb{R}^n}u^*(x-y)v_*(x)dx\right)^qdy\right)^{\frac{1}{q}}\left(\int_{\mathbb{R}^n}\left(w^*(y)\right)^{q'}dy\right)^{\frac{1}{q'}}\\
&=\|u^{*}\ast v_{*}\|_{L^{q}(\mathbb{R}^n)},
\end{split}\end{equation}
where the function $\bar{v}$ is defined by $\bar{v}(x):=v(-x)$.

\medskip

Now, assume that $\{f_j\}$ is a minimizing sequence for the minimizing problem \eqref{minimum}. Then, by \eqref{eq2}, we have
$$\|f_j\|_{L^p(\partial\mathbb{R}^n_+)}=\|f_j^*\|_{L^p(\partial\mathbb{R}^n_+)}=1$$
and
\begin{equation*}\begin{split}
\|Tf_j\|_{L^q(\mathbb{R}^n_+)}^q&=\int_0^\infty\|K_{x_n}\ast f_j\|_{L^q(\mathbb{R}^{n-1})}^qdx_n\\
&\geq\int_0^\infty\|K_{x_n}\ast f_j^*\|_{L^q(\mathbb{R}^{n-1})}^qdx_n\\
&=\|Tf_j^*\|_{L^q(\mathbb{R}^n_+)}^q.
\end{split}\end{equation*}
Thus it follows that $\{f_j^*\}$ is also a minimizing sequence. Hence, from now on, we may assume that $\{f_j\}$ is a sequence of nonnegative radial decreasing function.

From this, we shall write $f_j(x)$ by $f_j(|x|)$ or $f_j(r)$ where $r=|x|$. By the normalization $\|f_j\|_{L^p(\partial\mathbb{R}^n_+)}=1$, we have
$$1=\omega_{n-2}\int_{0}^\infty f_j^p(r)r^{n-2}dr\geq \nu_{n-1}f_j^p(R)R^{n-1}$$
for any $R>0$, where $\omega_{n-2}$ and $\nu_{n-1}$ denote the surface area of the unit sphere and the volume of the unit ball in $\mathbb{R}^{n-1}$, respectively. The above estimate implies that there exists some constant $C_{n,p}$ independent of $j$ such that for any $R>0$,
$$0\leq f_j(R)\leq C_{n,p}R^{-\frac{n-1}{p}}, \qquad j=1,2,\cdots.$$

\medskip

To continue our proof, we need the following lemma to modify the minimizing sequence so that $f_j$ will not converge to the trivial function as $j$ goes to infinity. For $\alpha\in(0,n)$ and $\beta=0$, the following lemma was first obtained by Lieb \cite{Lieb} in the whole space. Subsequently, Dou and Zhu \cite{DZ2}, Ng\^{o} and Nguyen \cite{Ngo1} extended it to $\alpha>n$, $\beta=0$.
\begin{lem}\label{lemsub}
Suppose that $f \in L^p(\partial\mathbb{R}^{n}_+)$ is non-negative, radially symmetric, and
\begin{equation}\label{ffeiling}
f(y)\leq \varepsilon |y|^{-\frac{n-1}{p}}
\end{equation}
for all $y \in \partial\mathbb{R}^n_+$. Then for any $p_1\in (0, p)$, there exists a constant $C_{n,\alpha,\beta,p,p_{1}}>0$, independent of
$f$ and $\varepsilon$, such that
$$\|Tf\|_{L^q(\mathbb{R}^{n}_+)}\geq C_{n,\alpha,\beta,p,p_{1}}\varepsilon^{1-\frac{p}{p_1}}\|f\|^{\frac{p}{p_1}}_{L^p(\partial\mathbb{R}^{n}_+)}.$$
\end{lem}
\begin{proof}
Define $F: \mathbb{R}\rightarrow\mathbb{R}$ by
$$F(t)=e^{\frac{(n-1)t}{p}}f(e^t).$$
It follows from \eqref{ffeiling} that
\begin{equation}\label{lemmad}
\omega_{n-2}^\frac{1}{p}\|F\|_{L^p(\mathbb{R})}=\|f\|_{L^p(\partial\mathbb{R}^{n}_+)} \quad \text{and} \quad \|F\|_{L^\infty(\mathbb{R})}\leq \varepsilon.
\end{equation}

For any $x'\in \mathbb{R}^{n-1}$, since both $K_{x_n}(x')=\frac{x_n^\beta}{\left(|x'|^2+x_{n}^2\right)^{\frac{n-\alpha}{2}}}$ and $f$ are radially symmetric function, one can easily check that $Tf(x',x_n)$ is also radially symmetric w.r.t. $x'$.

Now we define $H: \mathbb{R}\times\mathbb{R}_+\rightarrow\overline{\mathbb{R}_{+}}$ by
$$H(t,x_n):=e^{\frac{nt}{q}}Tf(e^t,e^{t}x_n).$$
Then, it follows that
\begin{equation}\label{lemmacc}
\omega_{n-2}^\frac{1}{q}\|H\|_{L^q(\mathbb{R}^2_+)}=\|Tf\|_{L^q(\mathbb{R}^{n}_+)}.
\end{equation}
Let any vector lying on $\partial\mathbb{R}^{n}_+$ with length $l$ be denoted by $\vec{l}$, we can calculate that
\begin{equation*}\begin{split}
H(t,x_n)&=e^{\frac{nt}{q}}\int_{\partial\mathbb{R}^{n}_+}(e^tx_n)^\beta\Big|
|\vec{e^t}-y|^2+e^{2t}x_n^2\Big|^{\frac{\alpha-n}{2}}f(y)dy\\
&=e^{\left(\frac{n}{q}+\frac{\alpha-n}{2}+\beta\right)t}\int_{\partial\mathbb{R}^{n}_+}x_n^\beta
\left|e^t(1+x_n^2)+e^{-t}|y|^2-2(\vec{1}\cdot y)\right|^{\frac{\alpha-n}{2}}f(y)dy\\
&=e^{\left(\frac{n}{q}+\frac{\alpha-n}{2}+\beta\right)t}\int_{-\infty}^{+\infty}\int_{\mathbb{S}^{n-2}}x_n^\beta
\left|e^{t-s}(1+x_n^2)+e^{-(t-s)}-2(\vec{1}\cdot \xi)\right|^{\frac{\alpha-n}{2}}\\
&\ \ \  \cdot f(e^s)e^{\left(n-1+\frac{\alpha-n}{2}\right)s}d\xi ds
\end{split}\end{equation*}
Then, thanks to $\frac{n-1}{n}\frac{1}{p}=\frac{\alpha+\beta-1}{n}+\frac{1}{q}$, we can rewrite $H$ as follows:
$$H(t,x_n)=\int_{-\infty}^{+\infty}L(t-s,x_n)F(s)ds,$$
where $L(s,x_n)=e^{\left(\frac{n}{q}+\frac{\alpha-n}{2}+\beta\right)s}Z(s,x_n)$ with
\begin{equation*}
Z(s,x_n)=\begin{cases}
\int_{\mathbb{S}^{n-2}}x_n^\beta
\left|e^{s}(1+x_{n}^{2})+e^{-s}-2(\vec{1}\cdot \xi)\right|^{\frac{\alpha-n}{2}}d\xi, \quad &n\geq3,\\
\\
x_n^\beta\left[\left(e^s(1+x_n^2)+e^{-s}-2\right)^{\frac{\alpha-n}{2}}+\left(e^s(1+x_n^2)+e^{-s}+2\right)^{\frac{\alpha-n}{2}}\right], \quad &n=2.
\end{cases}
\end{equation*}
It is easy to check that
$$L(t,x_n)\sim e^{\left(\frac{n}{q}+\frac{\alpha-n}{2}+\beta\right)t}x_{n}^{\beta}\left(e^{t}\left(1+x_{n}^{2}\right)+e^{-t}\right)^{\frac{\alpha-n}{2}}$$
as $t^2+x_n^2\rightarrow+\infty$. Recall that $\frac{n}{q}+\alpha-n+\beta=\left(\frac{1}{p}-1\right)(n-1)$ with $p\in(0,1)$, we deduce that
\begin{equation}\label{lemmaabc}
\frac{1}{q}+\alpha-n+\beta>\frac{n}{q}+\alpha-n+\beta>0.
\end{equation}
Since $q\beta+1>0$, then for any $r<0$, we have
\begin{equation}\label{lemmadd}
\int_{\mathbb{R}}\left(\int_0^\infty L^q(t,x_n)dx_n\right)^{\frac{r}{q}}dt\leq C_{n,\alpha,\beta,p,r}<+\infty.
\end{equation}

For any $p_1\in(0,p)$, we can choose $r_1$ such that
$$\frac{1}{p_1}+\frac{1}{r_1}=1+\frac{1}{q}.$$
Since $p_1\in(0,1)$ and $q<0$, we obtain that
\begin{equation}\label{lemmaddd}
r_1<0\ \ \text{and}\ \ \frac{q}{r_1}>1.
\end{equation}
It follows from the reversed Young inequality (see Lemma 2.2 in Dou and Zhu \cite{DZ2}) that
$$\int_{\mathbb{R}}|H(t,x_n)|^qdt\leq\left(\int_{\mathbb{R}} |L(t,x_n)|^{r_1}dt\right)^{\frac{q}{r_1}}
\left(\int_{\mathbb{R}} |F(t)|^{p_1}dt\right)^{\frac{q}{p_1}}.$$

By \eqref{lemmad}, \eqref{lemmadd}, \eqref{lemmaddd} and the Minkowski inequality, we have
\begin{equation*}\begin{split}
\int_{\mathbb{R}_{+}^{2}}|H(t,x_n)|^qdtdx_n&\leq\int_0^\infty\left(\int_{\mathbb{R}} |L(t,x_n)|^{r_1}dt\right)^{\frac{q}{r_1}}dx_n
\left(\int_{\mathbb{R}} |F(t)|^{p_1}dt\right)^{\frac{q}{p_1}}\\
&\leq\left(\int_{\mathbb{R}}\left(\int_{0}^{\infty}|L(t,x_n)|^{q}dx_n\right)^{\frac{r_{1}}{q}}dt\right)^{\frac{q}{r_1}}
\left(\int_{\mathbb{R}} |F(t)|^{p_1}dt\right)^{\frac{q}{p_1}}\\
&\leq C_{n,\alpha,\beta,p,p_{1}}\left(\int_{\mathbb{R}} |F(t)|^{p_1-p+p}dt\right)^{\frac{q}{p_1}}\\
&\leq C_{n,\alpha,\beta,p,p_{1}}\|F\|_{L^\infty(\mathbb{R})}^{q\left(1-\frac{p}{p_1}\right)}\left(\int_{\mathbb{R}} |F(t)|^{p}dt\right)^{\frac{q}{p_1}}\\
&\leq C_{n,\alpha,\beta,p,p_{1}}\varepsilon^{q\left(1-\frac{p}{p_1}\right)}\left(\int_{\partial\mathbb{R}^{n}_{+}}|f(x)|^{p}dx\right)^{\frac{q}{p_1}}.\\
\end{split}\end{equation*}
By \eqref{lemmacc}, we conclude that there exists some constant $C_{n,\alpha,\beta,p,p_{1}}>0$ such that
$$\|Tf\|_{L^q(\mathbb{R}^{n}_+)}\geq C_{n,\alpha,\beta,p,p_{1}}\varepsilon^{1-\frac{p}{p_1}}\|f\|^{\frac{p}{p_1}}_{L^p(\partial\mathbb{R}^{n}_+)}.$$
This completes the proof of Lemma \ref{lemsub}.
\end{proof}

\medskip

For convenience, denote $e_1:=(1,0,\cdots,0)\in \partial\mathbb{R}^n_+$, and define
$$a_j :=\sup_{\lambda>0}\lambda^{\frac{n-1}{p}}f_j(\lambda e_1).$$

Notice that, for any $y\in\partial\mathbb{R}^n_+$,
$$f_j(y) = f_j(|y|e_1) = |y|^{-\frac{n-1}{p}}|y|^{\frac{n-1}{p}}f_j(|y| e_1)\leq a_j|y|^{-\frac{n-1}{p}},$$
and $\|Tf_j\|_{L^q(\mathbb{R}^{n}_+)}\rightarrow C^{\ast}_{n,\alpha,\beta,p}<+\infty$. We can deduce from Lemma \ref{lemsub} that, there is a constant $c_{0}>0$ independent of $j$ such that
\begin{equation}\label{genggaif0}
a_j\geq2c_0>0, \qquad \forall \,\, j=1,2,\cdots.
\end{equation}
Define
$$f_j^{\lambda}(y)=\lambda^{\frac{n-1}{p}}f_j(\lambda y).$$
Note that $\frac{n-1}{n}\frac{1}{p}=\frac{\alpha+\beta-1}{n}+\frac{1}{q}$, it is easy to verify that
\begin{equation}\label{genggaif1}
\|f_j^\lambda\|_{L^p(\partial \mathbb{R}^n_+)}=\|f_j\|_{L^p(\partial \mathbb{R}^n_+)}, \qquad \|T(f_j^\lambda)\|_{L^q(\mathbb{R}^n_+)}=\|T(f_j)\|_{L^q(\mathbb{R}^n_+)}.
\end{equation}

For each $j$, by \eqref{genggaif0}, we can choose $\lambda_j$ so that $f_j^{\lambda_j}(e_{1})\geq c_0$. Due to \eqref{genggaif1}, we know that $\{f_j^{\lambda_j}\}^\infty_{j=1}$ is also a minimizing sequence for the minimizing problem \eqref{minimum}. Therefore, we can further assume that there is a nonnegative, radially symmetric and nonincreasing minimizing sequence $\{f_j\}^\infty_{j=1}$ with $\|f_j\|_{L^p(\partial \mathbb{R}^n_+)} =1$ and $f_j(e_1)\geq c_0>0$. Similar to Lieb's argument in Lemma 2.4 in \cite{Lieb}, we know, up to a subsequence, that $f_j\rightarrow f_0$ a.e. in $\partial \mathbb{R}^n_+$.

Now we are to show that the function $f_0$ is indeed a minimizer for \eqref{minimum}.

For all $x\in\mathbb{R}^n_+$, there exists a constant $C>0$ independent of $j$ such that
\begin{equation}\label{zhengshif}
Tf_j(x)\geq C\int_{y\in\partial\mathbb{R}^{n}_{+}, \,|y|\leq1}x_{n}^{\beta}|x-y|^{\alpha-n}dy\geq Cx_n^\beta\left(1+|x|^{\alpha-n}\right):=h(x).
\end{equation}
Let
$$k(x):=\liminf_{j\rightarrow\infty}Tf_j(x).$$
By \eqref{zhengshif}, for any $x\in\mathbb{R}^n_+$, we have
$$k(x)\geq h(x)>0.$$
It follows from Fatou's lemma that
\begin{equation*}\begin{split}
&-\int_{\mathbb{R}^n_+}k^q(x)dx=\int_{\mathbb{R}^n_+}\liminf_{j\rightarrow\infty}\left[-(Tf_j)^q(x)\right]dx \\
&\qquad\qquad\qquad\,\, \leq\liminf_{j\rightarrow\infty}\int_{\mathbb{R}^n_+}\left[-(Tf_j)^q(x)\right]dx \\
&\qquad\qquad\qquad\,\, =-\left(C^*_{n,\alpha,\beta,p}\right)^q.
\end{split}\end{equation*}
Thus we can infer that
\begin{equation}\label{gujik}
\int_{\mathbb{R}^n_+}k^q(x)dx\geq\left(C^*_{n,\alpha,\beta,p}\right)^q>0.
\end{equation}
This implies that $meas\{x\in\mathbb{R}^n_+ \mid 0<k(x)<+\infty\}>0$. Hence we can take a point $x^1\in\mathbb{R}^n_+$, and extract a subsequence of $Tf_j$ (still denoted by $Tf_j$) such that
$$\lim_{j\rightarrow\infty}Tf_j(x^1)=a_1\in(0,+\infty).$$
Repeating the above arguments and extracting a subsequence of $Tf_j$ if necessary, we can choose a point $x^2\in\mathbb{R}^n_+$ such that $x^2\neq x^1$ and $(x^2)_{n}>(x^1)_{n}>0$,
$$\lim_{j\rightarrow\infty}Tf_j(x^2)=a_2\in (0,+\infty).$$
Then, there exists a constant $C_1>0$ such that $Tf_j(x^i)\leq C_1$ for $i=1,2$ and all $j\geq1$. Therefore, for any $j\geq1$,
\begin{equation*}\begin{split}
&\quad \left[(x^1)_{n}\right]^\beta|x^1-x^2|^{\alpha-n}\int_{\partial\mathbb{R}^n_+}f_j(y)dy \\
&\leq\left[(x^1)_{n}\right]^\beta\max\left\{1,2^{\alpha-n-1}\right\}\int_{\partial\mathbb{R}^n_+}|x^1-y|^{\alpha-n}f_j(y)dy \\
&\ \ +\left[(x^2)_{n}\right]^\beta\max\left\{1,2^{\alpha-n-1}\right\}\int_{\partial\mathbb{R}^n_+}|x^2-y|^{\alpha-n}f_j(y)dy \\
&\leq\max\left\{1,2^{\alpha-n-1}\right\}\left(Tf_j(x_1)+Tf_j(x_2)\right)\\
&\leq2C_1\max\left\{1,2^{\alpha-n-1}\right\},\\
\end{split}\end{equation*}
where we have used the following point-wise inequality:
$$|u+v|^{\alpha-n}\leq\max\left\{1,2^{\alpha-n-1}\right\}\left(|u|^{\alpha-n}+|v|^{\alpha-n}\right), \qquad \forall \, u,v\in\mathbb{R}^n.$$
Thus there exists a constant $C_2>0$ such that, for all $j\geq1$,
\begin{equation}\label{zhengshifw1}
\int_{\partial\mathbb{R}^n_+}f_j(y)dy\leq C_2.
\end{equation}

For any $r>2|x^1|$ and $\frac{3}{4}r\leq |y|\leq r$, one can infer that $|x^1-y|\geq \frac{|y|}{3}$. Then, it follows immediately that, for any $j\geq1$,
\begin{equation*}\begin{split}
C_1&\geq Tf_{j}(x^{1})\geq\left[(x^1)_{n}\right]^\beta\int_{y\in\partial\mathbb{R}^{n}_{+}, \, \frac{3}{4}r\leq |y|\leq r}|x^1-y|^{\alpha-n}f_j(y)dy\\
&\geq 3^{n-\alpha}\left[(x^1)_{n}\right]^\beta f_j(r)r^{\alpha-1}\int_{\frac{3}{4}\leq |y|\leq 1}|y|^{\alpha-n}dy.
\end{split}\end{equation*}
Consequently, for any $r>2|x^1|$, there exists a constant $C_3>0$ such that
$$f_j(r)\leq C_{3}r^{1-\alpha}, \qquad \forall \, j\geq1.$$

Recall that $\frac{n-1}{n}\frac{1}{p}-\frac{1}{q}-\frac{\alpha+\beta-1}{n}=0$ and $q\beta+1>0$, for $R>2|x^1|$, we have
\begin{equation}\label{zhengshifw2}\begin{split}
\int_{y\in\partial\mathbb{R}^{n}_{+}, \, |y|\geq R}f^p_j(y)dy&\leq C_3^p\int_{y\in\partial\mathbb{R}^{n}_{+}, \, |y|\geq R}|y|^{p(1-\alpha)}dy\\
&= \frac{\omega_{n-2}C_3^p}{p(\alpha-1)-n+1} R^{p(1-\alpha)+n-1}\\
&= \frac{\omega_{n-2}C_3^p}{p(\alpha-1)-n+1} R^{p(\beta+\frac{n}{q})}.
\end{split}\end{equation}
By \eqref{zhengshifw1}, we have
\begin{equation}\label{zhengshifw3}
\int_{\{f_j\geq R\}}f^p_j(y)dy\leq R^{p-1}\int_{\partial\mathbb{R}^n_+}f_j(y)dy\leq C_2 R^{p-1}.
\end{equation}
For any $\varepsilon>0$, by \eqref{zhengshifw2} and \eqref{zhengshifw3}, one can choose $R_{\varepsilon}>2|x^1|$ sufficiently large such that, for any $R\geq R_{\varepsilon}$,
$$\int_{y\in\partial\mathbb{R}^{n}_{+}, \, |y|\geq R}f^p_j(y)dy<\frac{\varepsilon}{2}, \qquad \int_{\{f_j\geq R\}}f^p_j(y)dy<\frac{\varepsilon}{2}.$$
For every $j\geq1$ and $R\geq R_{\varepsilon}$, define
$$f_{j,R}(y):=\min\{f_j(y), R\}.$$
Then, we have, for any $j\geq1$ and $R\geq R_{\varepsilon}$,
\begin{equation*}\begin{split}
\int_{y\in\partial\mathbb{R}^{n}_{+}, \, |y|\leq R}\left[f_{j,R}(y)\right]^{p}dy&\geq \int_{\{|y|\leq R\}\cap\{f_j\leq R\}}f^p_j(y)dy\\
&= 1-\int_{\{|y|\leq R\}\cap\{f_j>R\}}f^p_j(y)dy-\int_{\{|y|> R\}}f^p_j(y)dy\\
&\geq1-\varepsilon.
\end{split}\end{equation*}

For arbitrarily given $R\geq R_{\varepsilon}$, it follows from the dominated convergence theorem that
$$1-\varepsilon\leq\lim_{j\rightarrow\infty}\int_{y\in\partial\mathbb{R}^{n}_{+}, \, |y|\leq R}\left[f_{j,R}(y)\right]^{p}dy
=\int_{y\in\partial\mathbb{R}^{n}_{+}, \, |y|\leq R}\left(\min\left\{f_{0}(y), R\right\}\right)^{p}dy.$$
Therefore, by letting $R\rightarrow\infty$, we get, for any $\varepsilon>0$,
$$\int_{\partial\mathbb{R}^{n}_{+}}f_{0}^{p}(y)dy\geq1-\varepsilon,$$
which yields immediately that $\int_{\partial\mathbb{R}^{n}_{+}}f_{0}^{p}(y)dy\geq1$. On the other hand, by Fatou's lemma, we have $\int_{\partial\mathbb{R}^{n}_{+}}f_{0}^{p}(y)dy\leq1$. Hence, we arrive at
$$\int_{\partial\mathbb{R}^{n}_{+}}f_{0}^{p}(y)dy=1.$$

It remains to show that $\|Tf_0\|_{L^q( \mathbb{R}^n_+)}=C^*_{n,\alpha,\beta,p}$. By Fatou's Lemma, we have
\begin{equation}\label{zhmend}
Tf_0(x)\leq{\lim\inf}_{j\rightarrow\infty}Tf_j(x)=k(x), \qquad \forall \, x\in\mathbb{R}^{n}_{+}.
\end{equation}
By \eqref{gujik} and \eqref{zhmend}, we arrive at
$$C^*_{n,\alpha,\beta,p}\geq \left(\int_{\mathbb{R}^n_+}k^q(x)dx\right)^{\frac{1}{q}} \geq
\left(\int_{\mathbb{R}^n_+}(Tf_0(x))^qdx\right)^{\frac{1}{q}}\geq C^*_{n,\alpha,\beta,p},$$
which implies that $f_0$ is a minimizer for the minimizing problem \eqref{minimum}.

This completes the proof of Theorem \ref{theorem2sub}.
\end{proof}

\section{The proof of Theorem \ref{theoremfen}}

In this section, we consider the following Euler-Lagrange equation for extremal functions to inequality \eqref{dengjiabudengshi}:
\begin{equation}\label{Euleraa}
f^{p-1}(y)=\int_{\mathbb{R}^n_+}\frac{x_n^\beta}{(|x'-y'|^2+x_n^2)^{\frac{n-\alpha}{2}}}(Tf(x))^{q-1}dx, \qquad \forall \, y\in\partial\mathbb{R}^{n}_{+}.
\end{equation}

Let $u(y)=f^{p-1}(y)$, $v(x)=x_{n}^{-\beta}Tf(x)$, $-\theta=\frac{1}{p-1}$, and $-k=q-1$. Euler-Lagrange equation \eqref{Euleraa} can be rewritten as the following integral system£º
\begin{equation}\label{intsp}\begin{cases}
u(y)=\int_{\mathbb{R}^n_+}\frac{x_n^{\beta(1-k)}}{|x-y|^{n-\alpha}}  v^{-k}(x)dx, \quad y\in\partial\mathbb{R}^n_+,\\ \\
v(x)=\int_{\partial\mathbb{R}^n_+}\frac{1}{|x-y|^{n-\alpha}}u^{-\theta}(y) dy, \quad x\in\mathbb{R}^n_+,
\end{cases}\end{equation}
where $\frac{1}{k-1}=\frac{n-1}{n}(\frac{\alpha+\beta-n}{n-1}-\frac{1}{\theta-1})$ with $\beta(1-k)+1>0$. The definition of $v(x)$ can be extended to $\overline{\mathbb{R}^{n}_{+}}$.

\begin{thm}\label{theoremfenlei}
Let $(u,v)$ be a pair of positive Lebesgue measurable solutions on $\partial\mathbb{R}^n_{+}\times \mathbb{R}^n_{+}$ satisfying the system \eqref{intsp}. Assume that
$$k\leq\frac{n+\alpha+2\beta}{\alpha+2\beta-n}, \qquad \theta\leq\frac{n+\alpha-2}{\alpha-n}.$$
Then, $u$ and $v$ must take the following form:
\begin{equation*}
u(y)=c_1\left(\frac{d}{1+d^{2}|y-z_0|^2}\right)^{\frac{n-\alpha}{2}}, \quad v(y,0)=c_2\left(\frac{d}{1+d^{2}|y-z_0|^2}\right)^{\frac{n-\alpha}{2}}, \qquad \forall \, y\in\partial\mathbb{R}^{n}_{+}
\end{equation*}
for some $z_0\in\partial \mathbb{R}^n_{+}$, $c_1>0$, $c_2>0$ and $d>0$. Furthermore, we have
$$k=\frac{n+\alpha+2\beta}{\alpha+2\beta-n},\ \ \ \theta=\frac{n+\alpha-2}{\alpha-n}.$$
Moreover, in such conformally invariant case, the value of the best constant in inequality \eqref{dengjiabudengshi} is
$$C^*_{n,\alpha,\beta,p}=\left(n\nu_n\right)^{-\frac{n+\alpha-2}{2(n-1)}}\left(\int_{ B_1}\left|\int_{\partial B_1}\left(\frac{1-|\xi|^2}{2}\right)^\beta|\eta-\xi|^{\alpha-n}d\eta\right|^{\frac{2n}{n-\alpha-2\beta}}d\xi\right)^{\frac{n-\alpha-2\beta}{2n}},$$
where $\nu_{n}$ denotes the volume of the unit ball in $\mathbb{R}^{n}$.
\end{thm}

By applying some new ideas, we classify positive Lebesgue measurable solutions to the system \eqref{intsp} \emph{without lifting the regularity of Lebesgue measurable solutions} and hence prove Theorem \ref{theoremfenlei}. Being different from the arguments in \cite{L}, we dilate the spheres $S_\lambda$ from the neighborhood of singular points to the limiting radius $\bar{\lambda}$. By exploiting the \emph{Spherically Narrow Region Maximum Principle} in integral forms, we use a slight variant of the method of moving spheres to classify all positive Lebesgue measurable solutions to integral system \eqref{intsp}.

\medskip

To this end, we give some necessary notations. For any $\lambda>0$, denote
$$B_{\lambda}(x):=\{y\in \mathbb{R}^n \mid |y-x|<\lambda,\ x\in \mathbb{R}^n\},$$
$$B_{\lambda}^{n-1}(x):=\{y\in \partial\mathbb{R}^n_{+} \mid \ |y-x|<\lambda,\ x\in \partial\mathbb{R}^n_{+}\},$$
$$B_{\lambda}^{+}(x):=\{y=(y_1,y_2,\cdots,y_n)\in B_{\lambda}(x) \mid \ y_n>0,\ x\in \partial\mathbb{R}^n_{+}\}.$$

For $z\in\partial\mathbb{R}^n_{+}$ and $\lambda>0$, set
$$x^{z,\lambda}:=\frac{\lambda^2(x-z)}{|x-z|^2}+z, \qquad y^{z,\lambda}:=\frac{\lambda^2(y-z)}{|y-z|^2}+z,$$
where $x\in \mathbb{R}^n_{+}$ and $y\in \partial\mathbb{R}^n_{+}\setminus\{z\}$. For any such $x$, $y$ and $z$, we have
\begin{equation}\label{equ2.1}
|x^{z,\lambda}-y^{z,\lambda}|=\frac{\lambda^2|x-y|}{|x-z||y-z|},
\end{equation}
\begin{equation}\label{equ2.2}
|y-z||x-y^{z,\lambda}|=|x-z||x^{z,\lambda}-y|,
\end{equation}
and
\begin{equation}\label{equ2.3}
\left(x^{z,\lambda}\right)_n=\left(\frac{\lambda}{|x-z|}\right)^2x_n.
\end{equation}
Let $(u,v)$ be a pair of positive functions defined on $\partial\mathbb{R}^n_{+}\times \mathbb{R}^n_{+}$. Define the Kelvin-type transforms
\begin{equation*}
u_{z,\lambda}(y)=\left(\frac{\lambda}{|y-z|}\right)^{n-\alpha}u(y^{z,\lambda}), \qquad \forall \, y\in \partial\mathbb{R}^n_{+}\setminus\{z\}
\end{equation*}
and
\begin{equation*}
v_{z,\lambda}(x)=\left(\frac{\lambda}{|x-z|}\right)^{n-\alpha}v(x^{z,\lambda}), \qquad \forall \, x\in \mathbb{R}^n_{+}.
\end{equation*}

\begin{lem}\label{lemma1}
Assume that $(u,v)$ is a pair of positive Lebesgue measurable solutions satisfying the system \eqref{intsp}. Then, for any $z\in\partial\mathbb{R}^n_{+}$ and $\lambda>0$,
\begin{equation}\begin{split}\label{equ11}
&\quad u(y)-u_{z,\lambda}(y)\\
&=\int_{B^{+}_{\lambda}(z)}K(z,\lambda,y,x)x_{n}^{\beta(1-k)}\left(\left(\frac{\lambda}{|x-z|}\right)^{n+\alpha+2\beta-(\alpha+2\beta-n)k}
v^{-k}_{z,\lambda}(x)-v^{-k}(x)\right)dx\\
\end{split}\end{equation}
for all $y\in \partial\mathbb{R}^n_{+}\setminus\{z\}$ and
\begin{equation}\label{equ12}\begin{split}
&\quad v(x)-v_{z,\lambda}(x)\\
&=\int_ {B^{n-1}_{\lambda}(z)}K(z,\lambda,y,x)\left(\left(\frac{\lambda}{|y-z|}\right)^{n+\alpha-2-(\alpha-n)\theta}
u^{-\theta}_{z,\lambda}(y)-u^{-\theta}(y)\right)dy\\
\end{split}\end{equation}
for all $x\in\mathbb{R}^n_{+}$, where
$$K(z,\lambda,y,x):=\left(\frac{\lambda}{|x-z|}\right)^{n-\alpha}\frac{1}{|x^{z,\lambda}-y|^{n-\alpha}}-\frac{1}{|x-y|^{n-\alpha}}.$$
Furthermore, for any $z\in\partial\mathbb{R}^n_{+}$ and $\lambda>0$,
\begin{equation}\label{equ224}
K(z,\lambda,y,x)>0, \qquad \forall \, x\in B^{+}_{\lambda}(z), \ y\in B^{n-1}_{\lambda}(z)\setminus\{z\}.
\end{equation}
\end{lem}
\begin{proof}
We first prove \eqref{equ11}. For $z\in\partial\mathbb{R}^n_{+}$ and $\lambda>0$, using the change of variable $x\mapsto x^{z,\lambda}$ and \eqref{equ2.3}, we obtain
\begin{equation*}
\int_{\mathbb{R}^{n}_{+}\setminus B_{\lambda}^{+}(z)}\frac{x_n^{\beta(1-k)}v^{-k}(x)}{|x-y|^{n-\alpha}}dx=\int_{B_{\lambda}^{+}(z)}\left(\frac{\lambda}{|x-z|}\right)^{2n+2\beta(1-k)}
\frac{x_n^{\beta(1-k)} v^{-k}(x^{z,\lambda})}{|x^{z,\lambda}-y|^{n-\alpha}}dx.
\end{equation*}
For any $y\in \partial \mathbb{R}^{n}_{+}$, by \eqref{intsp}, one can write
\begin{equation}\begin{split}\label{lemmadeduce1}
u(y)&=\int_{B_{\lambda}^{+}(z)}\frac{x_n^{\beta(1-k)}
v^{-k}(x)}{|x-y|^{n-\alpha}}dx+\int_{B_{\lambda}^{+}(z)}\left(\frac{\lambda}{|x-z|}\right)^{2n+2\beta(1-k)}\frac{x_n^{\beta(1-k)} v^{-k}(x^{z,\lambda})}{|x^{z,\lambda}-y|^{n-\alpha}}dx\\
&=\int_{B_{\lambda}^{+}(z)}\frac{x_n^{\beta(1-k)}v^{-k}(x)}{|x-y|^{n-\alpha}}dx+\int_{B_{\lambda}^{+}(z)}
\left(\frac{\lambda}{|x-z|}\right)^{2n+2\beta-(\alpha+2\beta-n)k}\frac{x_n^{\beta(1-k)}v^{-k}_{z,\lambda}(x)}{|x^{z,\lambda}-y|^{n-\alpha}}dx.\\
\end{split}\end{equation}
Consequently, using \eqref{equ2.1} and \eqref{equ2.2}, we have, for any $y\in \partial \mathbb{R}^{n}_{+}\setminus\{z\}$,
\begin{equation}\begin{split}\label{lemmadeduce2}
u_{z,\lambda}(y)&=\left(\frac{\lambda}{|y-z|}\right)^{n-\alpha}\int_{B_{\lambda}^{+}(z)}\frac{x_n^{\beta(1-k)}v^{-k}(x)}{|x-y^{z,\lambda}|^{n-\alpha}}dx\\
&\ \ +\left(\frac{\lambda}{|y-z|}\right)^{n-\alpha}\int_{B_{\lambda}^{+}(z)}\left(\frac{\lambda}{|x-z|}\right)^{2n+2\beta
-(\alpha+2\beta-n)k}\frac{x_n^{\beta(1-k)}v^{-k}_{z,\lambda}(x)}{|x^{z,\lambda}-y^{z,\lambda}|^{n-\alpha}}dx\\
&=\int_{B_{\lambda}^{+}(z)}\left(\frac{\lambda}{|x-z|}\right)^{n-\alpha}\frac{x_n^{\beta(1-k)}v^{-k}(x)}{|x^{z,\lambda}-y|^{n-\alpha}}dx\\
&\ \ +\int_{B_{\lambda}^{+}(z)}\left(\frac{\lambda}{|x-z|}\right)^{n+\alpha+2\beta-(\alpha+2\beta-n)k}\frac{x_n^{\beta(1-k)}
v^{-k}_{z,\lambda}(x)}{|x-y|^{n-\alpha}}dx.
\end{split}\end{equation}
Combining \eqref{lemmadeduce1} with \eqref{lemmadeduce2}, we deduce that, for all $y\in \partial\mathbb{R}^n_{+}\setminus\{z\}$,
\begin{equation*}\begin{split}
&\quad u(y)-u_{z,\lambda}(y)\\
&=\int_{B^{+}_{\lambda}(z)}K(z,\lambda,y,x)x_{n}^{\beta(1-k)}\left(\left(\frac{\lambda}{|x-z|}\right)^{n+\alpha+2\beta-(\alpha+2\beta-n)k}
v^{-k}_{z,\lambda}(x)-v^{-k}(x)\right)dx,
\end{split}\end{equation*}
where
\begin{equation}\label{kernelp}
K(z,\lambda,y,x):=\left(\frac{\lambda}{|x-z|}\right)^{n-\alpha}\frac{1}{|x^{z,\lambda}-y|^{n-\alpha}}-\frac{1}{|x-y|^{n-\alpha}}.
\end{equation}

Similarly, for any $x\in\mathbb{R}^n_{+}$, one can derive that
\begin{equation*}\begin{split}
&\quad v(x)-v_{z,\lambda}(x)\\
&=\int_ {B^{n-1}_{\lambda}(z)}K(z,\lambda,y,x)\left(\left(\frac{\lambda}{|y-z|}\right)^{n+\alpha-2-(\alpha-n)\theta}
u^{-\theta}_{z,\lambda}(y)-u^{-\theta}(y)\right)dy.
\end{split}\end{equation*}

Next, we show that
$$K(z,\lambda,y,x)>0, \qquad  \forall \, x\in B^{+}_{\lambda}(z), \ y\in B^{n-1}_{\lambda}(z)\setminus\{z\}.$$
By \eqref{kernelp}, it suffices to verify that
$$|x-z|^{n-\alpha}|x^{z,\lambda}-y|^{n-\alpha}<\lambda^{n-\alpha}|x-y|^{n-\alpha},$$
which is equivalent to prove that
\begin{equation*}\begin{split}
&|x-z|^{2}|x^{z,\lambda}-y|^{2}-\lambda^{2}|x-y|^{2}\\
&=|x-z|^2\left[\frac{\lambda^4}{|x-z|^2}-2\lambda^2\cdot\frac{(x-z)(y-z)}{|x-z|^2}
+|y-z|^2\right]-\lambda^2|x-y|^2\\
&=\left(\lambda^2-|x-z|^2\right)\left(\lambda^2-|y-z|^2\right)\\
&>0.
\end{split}\end{equation*}
This completes the proof of Lemma \ref{lemma1}.
\end{proof}

\medskip

We use the idea of Li \cite{L} to prove the following lemma.
\begin{lem}\label{lem2}
Assume that $(u,v)$ is a pair of positive Lebesgue measurable solutions on $\partial\mathbb{R}^n_{+}\times \mathbb{R}^n_{+}$ satisfying the system \eqref{intsp}, then

\noindent (i) $\int_{\partial \mathbb{R}^n_{+}}(1+|y|^{\alpha-n}) u^{-\theta}(y)dy<+\infty$ and $\int_{\mathbb{R}^n_{+}}x_n^{\beta(1-k)}(1+|x|^{\alpha-n}) v^{-k}(x)dx<+\infty$;

\noindent (ii) there exist constants $C_1\geq1$ and $C_2\geq1$ such that
$$\frac{1}{C_1}(1+|y|^{\alpha-n})\leq u(y)\leq C_1(1+|y|^{\alpha-n}),$$
$$\frac{1}{C_2}(1+|x|^{\alpha-n})\leq v(x)\leq C_2(1+|x|^{\alpha-n});$$

\noindent (iii) the following asymptotic properties hold:
$$\lim_{|y|\rightarrow\infty}\frac{u(y)}{|y|^{\alpha-n}}=\int_{\mathbb{R}^n_{+}}x_n^{\beta(1-k)}v^{-k}(x)dx<+\infty,$$ $$\lim_{|x|\rightarrow\infty}\frac{v(x)}{|x|^{\alpha-n}}=\int_{\partial\mathbb{R}^n_{+}}u^{-\theta}(y)dy<+\infty.$$

Furthermore, $(u,v)\in C(\partial\mathbb{R}^n_{+})\times C(\overline{\mathbb{R}^n_{+}})$.
\end{lem}
\begin{proof}
Since $(u,v)$ is a pair of positive Lebesgue measurable solutions of \eqref{intsp}, we have
\begin{equation*}
\text{meas}\{y\in\partial \mathbb{R}^n_{+} \mid u(y)<+\infty\}>0,\ \ \ \text{meas}\{x\in\mathbb{R}^n_{+} \mid v(x)<+\infty\}>0.
\end{equation*}
Moreover, there exist $R>1$ large enough and a measurable set $E$ such that
\begin{equation*}
E\subset\{x\in\mathbb{R}^n_{+} \mid v(x)<R\}\cap B_R^{+}(0)
\end{equation*}
with $meas(E)>\frac{1}{R}$.

For any $y\in\partial \mathbb{R}^n_{+}$, due to $\beta\geq0$ and $k>1$, we have
\begin{equation*}\begin{split}
u(y)&=\int_{\mathbb{R}^n_{+}}x_n^{\beta(1-k)}|x-y|^{\alpha-n} v^{-k}(x)dx\\
&\geq \int_{E}x_n^{\beta(1-k)}|x-y|^{\alpha-n}v^{-k}(x)dx\\
&\geq R^{-k+\beta(1-k)}\int_{E}|x-y|^{\alpha-n} dx.
\end{split}\end{equation*}
Since $\alpha>n$, there exists a constant $C_1\geq1$ such that
$$u(y)\geq\frac{1+|y|^{\alpha-n}}{C_1}, \qquad \forall \, y\in\partial \mathbb{R}^n_{+}.$$
Similarly, for any $x\in \mathbb{R}^n_{+}$, we have
\begin{equation}\label{vguji}
v(x)\geq\frac{1+|x|^{\alpha-n}}{C_2}.
\end{equation}
Thus we have obtained the left hand sides of the two inequalities in (ii).

\medskip

On the other hand, there exists a point $y^0\in\partial\mathbb{R}^{n}_{+}\setminus\{0\}$ such that
\begin{equation}\label{eq-a2}
  u(y^0)=\int_{\mathbb{R}^n_{+}}x_n^{\beta(1-k)}|x-y^0|^{\alpha-n}v^{-k}(x)dx<+\infty.
\end{equation}
Hence, one has
\begin{equation*}\begin{split}
&\quad \int_{\mathbb{R}^n_{+}}x_n^{\beta(1-k)}(1+|x|^{\alpha-n})v^{-k}(x)dx\\
&\leq C_{y^0}\int_{|x|<\frac{1}{2}|y^0|}x_n^{\beta(1-k)}|x-y^0|^{\alpha-n} v^{-k}(x)dx
+C_{y^0}\int_{|x|>2|y^0|}x_n^{\beta(1-k)}|x-y^0|^{\alpha-n} v^{-k}(x)dx \\
&\quad +\int_{\frac{1}{2}|y^0|\leq|x|\leq2|y^0|}x_n^{\beta(1-k)}(1+|x|^{\alpha-n})v^{-k}(x)dx,
\end{split}\end{equation*}
combining this with \eqref{vguji}, \eqref{eq-a2} and $(1-k)\beta+1>0$, we obtain the second formula in (i). The first formula in (i) can be proved similarly.

\medskip

By (i), we have, for any $y\in \partial\mathbb{R}^n_+$,
\begin{equation}\label{eq-a0}\begin{split}
\frac{u(y)}{1+|y|^{\alpha-n}}&=\int_{\mathbb{R}^n_{+}}\frac{|x-y|^{\alpha-n}}{1+|y|^{\alpha-n}}x_n^{\beta(1-k)}v^{-k}(x)dx\\
&\leq \int_{\mathbb{R}^n_{+}} x_n^{\beta(1-k)}(1+|x|^{\alpha-n}) v^{-k}(x)dx<+\infty.
\end{split}\end{equation}
Similarly, from (i), we infer that, for any $x\in \mathbb{R}^n_+$,
\begin{equation}\label{eq-a1}\begin{split}
\frac{v(x)}{1+|x|^{\alpha-n}}&=\int_{\partial\mathbb{R}^n_{+}}\frac{|y-x|^{\alpha-n}}{1+|x|^{\alpha-n}}u^{-\theta}(y)dy\\
&\leq \int_{\partial\mathbb{R}^n_{+}}(1+|y|^{\alpha-n})u^{-\theta}(y)dy<+\infty.
\end{split}\end{equation}
Then the right hand sides of the two inequalities in (ii) follow from \eqref{eq-a0} and \eqref{eq-a1}.

\medskip

Next, we will show the asymptotic behaviors of $u$ and $v$ near infinity in (iii).

For any $x\in\mathbb{R}^{n}_{+}$ with $|x|>1$ and $y\in\partial\mathbb{R}^{n}_{+}$ with $|y|>1$, by (i) and (ii), we have
\begin{equation*}\begin{split}
\frac{u(y)}{|y|^{\alpha-n}}&=|y|^{n-\alpha}\int_{\mathbb{R}^n_{+}}x_n^{\beta(1-k)}|x-y|^{\alpha-n}v^{-k}(x)dx\\
&\leq C \int_{\mathbb{R}^n_{+}}x_n^{\beta(1-k)}(1+|x|^{\alpha-n})v^{-k}(x)dx<+\infty
\end{split}\end{equation*}
and
\begin{equation*}\begin{split}
\frac{v(x)}{|x|^{\alpha-n}}&=|x|^{n-\alpha}\int_{\partial \mathbb{R}^n_{+}}|x-y|^{\alpha-n}u^{-\theta}(y)dy\\
&\leq C \int_{\partial\mathbb{R}^n_{+}}(1+|y|^{\alpha-n})u^{-\theta}(y)dy<+\infty.
\end{split}\end{equation*}
Hence, by Lebesgue's dominated convergence theorem, we arrive at (iii).

\medskip

Now, we are to show that $u$ and $v$ are continuous functions on $\partial\mathbb{R}^{n}_{+}$ and $\overline{\mathbb{R}^{n}_{+}}$ respectively. For the sake of simplicity, we only consider the continuity of $v$. The continuity of $u$ can be derived in quite similar way.

For any given $x^1\in\overline{\mathbb{R}^n_{+}}$, $x^2\in\overline{\mathbb{R}^n_{+}}$ and arbitrary $y\in\partial\mathbb{R}^n_{+}$, it is easy to verify that
\begin{equation*}
\Big||x^1-y|^{\alpha-n}-|x^2-y|^{\alpha-n}\Big|\leq
\begin{cases}
C|x^1-x^2|^{\alpha-n}, \qquad \text{if} \,\, 0<\alpha-n<1,\\ \\
C\left(1+|y|^{\alpha-n-1}\right)|x^1-x^2|, \qquad \text{if} \,\, \alpha-n\geq1.
\end{cases}
\end{equation*}
Then, it follows that
\begin{equation*}\begin{split}
|v(x^1)-v(x^2)|&=\left|\int_{\partial \mathbb{R}^n_{+}}|x^1-y|^{\alpha-n}u^{-\theta}(y)dy
-\int_{\partial \mathbb{R}^n_{+}}|x^2-y|^{\alpha-n}u^{-\theta}(y)dy\right| \\
&\leq\int_{\partial \mathbb{R}^n_{+}}\left||x^1-y|^{\alpha-n}-|x^2-y|^{\alpha-n}\right|u^{-\theta}(y)dy \\
&\leq C\max\left\{|x^1-x^2|^{\alpha-n},|x^1-x^2|\right\}\int_{\partial \mathbb{R}^n_{+}}\left(1+|y|^{\alpha-n}\right)u^{-\theta}(y)dy.
\end{split}\end{equation*}
This implies immediately that $v\in C(\mathbb{R}^{n}_{+})$. Thus we finish the proof of Lemma \ref{lem2}.
\end{proof}

\medskip

We first show that $u-u_{z,\lambda}$ and $v-v_{z,\lambda}$ are strictly positive in a small neighborhood of $z$.
\begin{lem}\label{lemaa}
Assume that $(u,v)$ is a pair of positive Lebesgue measurable solutions on $\partial\mathbb{R}^n_{+}\times \mathbb{R}^n_{+}$ satisfying the system \eqref{intsp}. Then, there exists a $\delta_{0}$ small enough such that, for any $0<\lambda\leq\delta_{0}$, there exist constants $C_3>0$ and $C_4>0$ such that
\begin{align*}
&u(y)-u_{z,\lambda}(y)>C_3>0, \qquad \forall \, y\in B^{n-1}_{\lambda^{2}}(z)\setminus \{z\},\\
&v(x)-v_{z,\lambda}(x)>C_4>0, \qquad \forall \, x\in B^+_{\lambda^{2}}(z).
\end{align*}
\end{lem}
\begin{proof}
Without loss of generality, let $z=0$. For any $y\in B^{n-1}_{\lambda^{2}}(0)\setminus \{0\}$, it is easy to check that $|y^{{0,\lambda}}|\geq1$. Therefore, by Lemma \ref{lem2}, we have
\begin{equation*}\begin{split}
u_{0,\lambda}(y)&=\left(\frac{\lambda}{|y|}\right)^{n-\alpha}u(y^{0,\lambda})\\
&\leq \left(\frac{\lambda}{|y|}\right)^{n-\alpha}\cdot 2C_1|y^{0,\lambda}|^{\alpha-n}\\
&\leq 2C_1\lambda^{\alpha-n}.
\end{split}\end{equation*}
Thus there exists a $\delta_{0}>0$ sufficiently small, one can deduce that, for any $0<\lambda\leq\delta_{0}$, there exists a constant $C_3>0$ such that
\begin{equation*}
u(y)-u_{0,\lambda}(y)\geq \inf_{{B^{n-1}_{\lambda^{2}}(0)}\setminus \{0\}}u(y)-2C_1\lambda^{\alpha-n}>C_3>0, \qquad \forall \, y\in B^{n-1}_{\lambda^{2}}(0)\setminus \{0\}.
\end{equation*}

Similarly, one can derive from Lemma \ref{lem2} that, for any $x\in B^{+}_{\lambda^{2}}(0)$,
\begin{equation*}\begin{split}
v_{0,\lambda}(x)&=\left(\frac{\lambda}{|x|}\right)^{n-\alpha}v(x^{0,\lambda})\\
&\leq \left(\frac{\lambda}{|x|}\right)^{n-\alpha}\frac{2C_2}{|x^{0,\lambda}|^{n-\alpha}}\\
&\leq \frac{2C_2}{\lambda^{n-\alpha}}.
\end{split}\end{equation*}
It follows that there exists a $\delta_{0}>0$ sufficiently small such that, for any $0<\lambda\leq\delta_{0}$, there exists a constant $C_4>0$ such that
\begin{equation*}\begin{split}
v(x)-v_{0,\lambda}(x)&\geq v(x)-\frac{2C_2}{\lambda^{n-\alpha}} \\
&\geq\inf_{B^{+}_{\lambda^{2}}(0)}v(x)-2C_2\lambda^{\alpha-n} \\
&\geq C_4>0, \qquad \forall \, x\in B^{+}_{\lambda^{2}}(0).
\end{split}\end{equation*}
This completes the proof of Lemma \ref{lemaa}.
\end{proof}

\medskip

The following lemma provides a starting point for us to move the spheres.
\begin{lem}\label{lemmastart}
For every $z\in\partial\mathbb{R}^n_{+}$, there exists $\epsilon_0(z)>0$ such that, for all $\lambda\in (0,\epsilon_0(z)]$,
\begin{equation}\label{qibu1}
u(y)\geq u_{z,\lambda}(y), \qquad  \forall \, y\in B^{n-1}_{\lambda}(z)\setminus \{z\},
\end{equation}
and
\begin{equation}\label{qibu2}
v(x)\geq v_{z,\lambda}(x), \qquad \forall \, x\in B_{\lambda}^{+}(z).
\end{equation}
\end{lem}
\begin{proof}
For any $z\in\partial\mathbb{R}^n_{+}$ and $\lambda>0$, define
$$B_{\lambda,u}^-=\{y\in B^{n-1}_{\lambda}(z)\setminus \{z\} \mid u(y)<u_{z,\lambda}(y)\}$$
and
$$B_{\lambda,v}^-=\{x\in B_\lambda^{+}(z) \mid v(x)<v_{z,\lambda}(x)\}.$$
We start the moving sphere procedure by showing that $B_{\lambda,u}^-=B_{\lambda,v}^-=\emptyset$ for sufficiently small $\lambda$.

For any $y \in B_{\lambda,u}^-$, by using \eqref{equ11} and the positivity of $K(z,\lambda,y,x)$, we obtain from $k\leq\frac{n+\alpha+2\beta}{\alpha+2\beta-n}$ that
\begin{equation}\begin{split}\nonumber\\
0&<u_{z,\lambda}(y)-u(y)\\
&=\int_{B^{+}_{\lambda}(z)}K(z,\lambda,y,x)x_{n}^{\beta(1-k)}\left(v^{-k}(x)-\left(\frac{\lambda}{|x-z|}\right)^{n+\alpha+2\beta-(\alpha+2\beta-n)k}
v^{-k}_{z,\lambda}(x)\right)dx\\
&\leq\int_{B_{\lambda,v}^-}K(z,\lambda,y,x)x_{n}^{\beta(1-k)}\left(v^{-k}(x)-\left(\frac{\lambda}{|x-z|}\right)^{n+\alpha+2\beta-(\alpha+2\beta-n)k}
v^{-k}_{z,\lambda}(x)\right)dx\\
&\leq\int_{B_{\lambda,v}^-}K(z,\lambda,y,x)x_{n}^{\beta(1-k)}\left(v^{-k}(x)-v^{-k}_{z,\lambda}(x)\right)dx\\
&\leq k\int_{B_{\lambda,v}^-}K(z,\lambda,y,x)x_{n}^{\beta(1-k)}v^{-k-1}(x)\left(v_{z,\lambda}(x)-v(x)\right)dx.
\end{split}\end{equation}

By Lemma \ref{lemaa}, we infer that
$$B_{\lambda,v}^-\subseteq B^+_\lambda(z)\setminus B^+_{\lambda^{2}}(z).$$
Since $v>0$ is continuous in $\overline{\mathbb{R}^{n}_{+}}$, we conclude that there exists a constant $C'>0$ independent of $\lambda$ such that, for any $y\in B_{\lambda,u}^{-}$ and $x\in B_{\lambda,v}^-$,
\begin{equation}\label{fvshangjie}\begin{split}
&\quad k v^{-k-1}(x)K(z,\lambda,y,x)\\
&=k v^{-k-1}(x)\left(\left(\frac{\lambda}{|x-z|}\right)^{n-\alpha}\frac{1}
{|x^{z,\lambda}-y|^{n-\alpha}}-\frac{1}{|x-y|^{n-\alpha}}\right)\\
&=k v^{-k-1}(x)\left(\left(\frac{|x-z|}{\lambda}\right)^{\alpha-n}
\left|\frac{\lambda^2(x-z)}{|x-z|^2}-(y-z)\right|^{\alpha-n}-|x-y|^{\alpha-n}\right)\\
&=k v^{-k-1}(x)\left(\left|\frac{\lambda(x-z)}{|x-z|}-\frac{(y-z)|x-z|}{\lambda}\right|^{\alpha-n}-|x-y|^{\alpha-n}\right)\\
&\leq C'\lambda^{\alpha-n}.
\end{split}\end{equation}

Then, for sufficiently small $\lambda$, it follows from \eqref{fvshangjie} that
\begin{equation}\begin{split}\label{equ20}
\|u_{z,\lambda}-u\|_{L^1(B_{\lambda,u}^-)}&\leq C'\lambda^{\alpha-n}|B_{\lambda,u}^-|\cdot\left\|x_{n}^{\beta(1-k)}\left(v_{z,\lambda}-v\right)\right\|_{L^1(B_{\lambda,v}^-)}\\
&\leq \frac{1}{2}\left\|x_{n}^{\beta(1-k)}\left(v_{z,\lambda}-v\right)\right\|_{L^1(B_{\lambda,v}^-)}.
\end{split}\end{equation}

Similarly, from \eqref{equ12}, the positivity of $K(z,\lambda,y,x)$, $\theta\leq\frac{n+\alpha-2}{\alpha-n}$ and $\beta(1-k)>-1$, we deduce that there exists a constant $C''>0$ independent of $\lambda$ such that, for sufficiently small $\lambda$,
\begin{equation}\begin{split}\label{equ21}
\left\|x_{n}^{\beta(1-k)}\left(v_{z,\lambda}-v\right)\right\|_{L^1(B_{\lambda,v}^-)}&\leq C''\left(\int_{B_{\lambda,v}^{-}}x_{n}^{\beta(1-k)}dx\right)\cdot\|u_{z,\lambda}-u\|_{L^1(B_{\lambda,u}^-)}\\
&\leq \frac{1}{2}\|u_{z,\lambda}-u\|_{L^1(B_{\lambda,u}^-)}.
\end{split}\end{equation}

From \eqref{equ20} and \eqref{equ21}, we conclude that there exists an $\epsilon_{0}(z)$ small enough such that, for any $0<\lambda\leq\epsilon_{0}(z)$,
$$\|u_{z,\lambda}-u\|_{L^{1}(B_{\lambda,u}^-)}=\left\|x_{n}^{\beta(1-k)}\left(v_{z,\lambda}-v\right)\right\|_{L^1(B_{\lambda,v}^-)}=0.$$
Therefore, the measure of $B_{\lambda,u}^-$ and $B_{\lambda,v}^-$ must be 0, and hence $B_{\lambda,u}^-=B_{\lambda,v}^-=\emptyset$ for any $\lambda\in(0,\epsilon_{0}(z)]$. This completes the proof of Lemma \ref{lemmastart}.
\end{proof}

\medskip

We now continue the moving spheres procedure as long as \eqref{qibu1} and \eqref{qibu2} hold. For each fixed $z\in \partial\mathbb{R}^n_+$, define
\begin{equation}\label{defn}
  \bar{\lambda}(z)=\sup\{\lambda>0 \mid u\geq u_{z,\mu}\,\, \text{in} \,\, B^{n-1}_{\mu}(z)\setminus\{z\},\ \
  v\geq v_{z,\mu}\,\, \text{in} \,\, B^+_{\mu}(z),\,\, \forall \,\, 0<\mu\leq\lambda\}.
\end{equation}
By Lemma \ref{lemmastart}, $\bar{\lambda}(z)$ is well-defined and $0<\bar{\lambda}(z)\leq+\infty$ for any $z\in \partial\mathbb{R}^n_+$.

\medskip

We need the following Lemma, which is crucial in our proof.
\begin{lem}\label{lemmasequali}
Assume $\bar{z}\in \partial\mathbb{R}^n_+$ satisfies $\bar{\lambda}(\bar{z})<+\infty$, then
\begin{equation}\label{start11}
u(y)=u_{\bar{z},\bar{\lambda}(\bar{z})}(y), \qquad \forall \, y\in B^{n-1}_{\bar{\lambda}}(\bar{z})\setminus\{\bar{z}\}
\end{equation}
and
\begin{equation}\label{start12}
v(x)=v_{\bar{z},\bar{\lambda}(\bar{z})}(x), \qquad \forall \, x\in B^{+}_{\bar{\lambda}}(\bar{z}).
\end{equation}
Furthermore, we must have
\begin{equation}\label{ktheta}
k=\frac{n+\alpha+2\beta}{\alpha+2\beta-n} \quad \text{and} \quad \theta=\frac{n+\alpha-2}{\alpha-n}.
\end{equation}
\end{lem}
\begin{proof}
Without loss of generality, we may assume that $\bar{z}=0$. For simplicity of notations, we denote $\bar{\lambda}:=\bar{\lambda}(0)$.

Suppose on the contrary that $u-u_{0,\bar{\lambda}}\geq0$ in $B^{n-1}_{\bar{\lambda}}(0)\setminus\{0\}$ and $v-v_{0,\bar{\lambda}}\geq0$ in $B^+_{\bar{\lambda}}(0)$ but at least one of them is not identically zero. Without loss of generality, we may assume that $v-v_{0,\bar{\lambda}}\geq0$ but $v-v_{0,\bar{\lambda}}$ is not identically zero in $B^+_{\bar{\lambda}}(0)$. Then we will get a contradiction with the definition \eqref{defn} of $\bar{\lambda}$.

We first show that
\begin{equation}\label{positive1}
  u(y)-u_{0,\bar{\lambda}}(y)>0, \,\,\,\,\,\, \forall \, y\in B^{n-1}_{\bar{\lambda}}(0)\setminus\{0\},
\end{equation}
and
\begin{equation}\label{positive2}
  v(x)-v_{0,\bar{\lambda}}(x)>0, \,\,\,\,\,\, \forall \, x\in B^{+}_{\bar{\lambda}}(0).
\end{equation}
Indeed, choose a point $x^{0}\in B^{+}_{\bar{\lambda}}(0)$ such that $v(x^{0})-v_{0,\bar{\lambda}}(x^{0})>0$, by continuity, there exists a small $\eta>0$ and a constant $c_{0}>0$ such that
\begin{equation}\label{vbigger}
B_{\eta}(x^{0})\subset B^+_{\bar{\lambda}}(0) \,\,\,\,\,\, \text{and} \,\,\,\,\,\,
v(x)-v_{0,\bar{\lambda}}(x)\geq c_{0}>0, \,\,\,\, \forall \, x\in B_{\eta}(x^{0}).
\end{equation}

For any $y\in B^{n-1}_{\bar{\lambda}}(0)\setminus\{0\}$, since $K(0,\bar{\lambda},y,x)>0$, one can derive from \eqref{equ11} that
\begin{equation}\label{equ11aa}\begin{split}
&\quad u(y)-u_{0,\bar{\lambda}}(y)\\
&=\int_{B^{+}_{\bar{\lambda}}(0)}K(0,\bar{\lambda},y,x)x_{n}^{\beta(1-k)}\left(\left(\frac{\bar{\lambda}}{|x|}\right)^{n+\alpha+2\beta-(\alpha+2\beta-n)k}
v^{-k}_{0,\bar{\lambda}}(x)-v^{-k}(x)\right)dx\\
&\geq\int_{B_\eta(x^0)}K(0,\bar{\lambda},y,x)x_{n}^{\beta(1-k)}\left(v^{-k}_{0,\bar{\lambda}}(x)-v^{-k}(x)\right)dx\\
&>0,
\end{split}\end{equation}
thus we arrive at \eqref{positive1}. By \eqref{equ12} and the positivity of $K(0,\bar{\lambda},y,x)$, one can immediately obtain \eqref{positive2}.

Pick a small $0<\delta_1<\frac{\bar{\lambda}}{2}$ and define
\begin{equation}\label{2-40}
m_{1}:=\inf_{y\in\overline{B^{n-1}_{\bar{\lambda}-\delta_1}(0)}\setminus\{0\}}\left(u-u_{0,\bar{\lambda}}\right)(y)>0.
\end{equation}
Then, by the continuity of $u-u_{0,\lambda}$ with respect to $\lambda$, there exists a smaller $0<\delta_2<\delta_{1}$ such that, for every $\lambda\in[\bar{\lambda},\bar{\lambda}+\delta_2]$,
\begin{equation}\label{udexiajie}
  u(y)-u_{0,\lambda}(y)\geq\frac{m_{1}}{2}>0, \qquad \forall \, y\in\overline{B^{n-1}_{\bar{\lambda}-\delta_1}(0)}\setminus\{0\}.
\end{equation}

Next, we will show that there exists a constant $m_2>0$ (to be determined later) such that, for all $\lambda\in[\bar{\lambda},\bar{\lambda}+\delta_2]$,
\begin{equation}\label{vdexiajie}
v(x)-v_{0,\lambda}(x)\geq m_2>0, \qquad \forall \, x\in B^+_{\bar{\lambda}-\delta_1}(0).
\end{equation}
By \eqref{equ12} and \eqref{udexiajie}, for any $x\in B^+_{\bar{\lambda}-\delta_1}(0)$, one can deduce that
\begin{equation}\label{equ12dd}\begin{split}
&\quad v(x)-v_{0,\lambda}(x)\\
&=\int_{B^{n-1}_{\lambda}(0)}K(0,\lambda,y,x)\left(\left(\frac{\lambda}{|y|}\right)^{n+\alpha-2-(\alpha-n)\theta}
u^{-\theta}_{0,\lambda}(y)-u^{-\theta}(y)\right)dy\\
&=\int_{B^{n-1}_{\bar{\lambda}-\delta_1}(0)}K(0,\lambda,y,x)\left(\left(\frac{\lambda}{|y|}\right)^{n+\alpha-2-(\alpha-n)\theta}
u^{-\theta}_{0,\lambda}(y)-u^{-\theta}(y)\right)dy\\
&\ \ \ +\int_{B^{n-1}_{\lambda}(0)
\setminus{B^{n-1}_{\bar{\lambda}-\delta_1}(0)}}K(0,\lambda,y,x)\left(\left(\frac{\lambda}{|y|}\right)^{n+\alpha-2-(\alpha-n)\theta}
u^{-\theta}_{0,\lambda}(y)-u^{-\theta}(y)\right)dy\\
&=:I_1(x)+I_2(x).
\end{split}\end{equation}

We first take $I_1(x)$ into account. For any $ y\in \overline{B_{\bar{\lambda}-\delta_1}^{n-1}(0)}\setminus\{0\}$ and $x\in B^+_{\bar{\lambda}-\delta_1}(0)$, we know that the kernel $K(0,\lambda,y,x)>0$ for any $\lambda\in[\bar{\lambda},\bar{\lambda}+\delta_2]$. From \eqref{udexiajie}, we derive that there exists a uniform constant $m_2>0$ depending only on $n$, $\alpha$, $\theta$, $\bar{\lambda}$ and $m_{1}$ such that
\begin{equation}\label{eq-a5}
  I_1(x)\geq 2m_2>0, \qquad \forall \, x\in B^+_{\bar{\lambda}-\delta_1}(0).
\end{equation}

Now we estimate $I_2(x)$. Since $u>0$ is a continuous function on $\partial\mathbb{R}^{n}_{+}$, it follows that, for any $\lambda\in[\bar{\lambda},\bar{\lambda}+\delta_2]$,
\begin{eqnarray}\label{eq-a3}
  &&\quad K(0,\lambda,y,x)\left(\left(\frac{\lambda}{|y|}\right)^{n+\alpha-2-(\alpha-n)\theta}
u^{-\theta}_{0,\lambda}(y)-u^{-\theta}(y)\right) \\
 \nonumber &&\geq -\left(\left|\frac{\lambda x}{|x|}-\frac{y|x|}{\lambda}\right|^{\alpha-n}-|x-y|^{\alpha-n}\right)u^{-\theta}(y) \\
 \nonumber &&\geq -\left(2\bar{\lambda}\right)^{\alpha-n}u^{-\theta}(y) \\
 \nonumber &&\geq -C,  \qquad \forall \, y\in B^{n-1}_{\lambda}(0)\setminus{B^{n-1}_{\bar{\lambda}-\delta_1}(0)}, \,\, \forall \, x\in B^+_{\bar{\lambda}-\delta_1}(0),
\end{eqnarray}
where the uniform constant $C>0$ depends only on $n$, $\alpha$, $\theta$ and $\bar{\lambda}$. Consequently, we have
\begin{equation}\label{eq-a4}
  I_{2}(x)\geq -C\left|B^{n-1}_{\lambda}(0)\setminus{B^{n-1}_{\bar{\lambda}-\delta_1}(0)}\right|
  =\widetilde{C}\left[\lambda^{n-1}-(\bar{\lambda}-\delta_1)^{n-1}\right], \qquad \forall \, x\in B^+_{\bar{\lambda}-\delta_1}(0),
\end{equation}
where the uniform constant $\widetilde{C}>0$ depends only on $n$, $\alpha$, $\theta$ and $\bar{\lambda}$.

Combining \eqref{eq-a4} with \eqref{eq-a5} and \eqref{equ12dd}, we derive that, by choosing $\delta_1>0$ and hence $\delta_2\in(0,\delta_{1})$ smaller if necessary, for any $\lambda\in[\bar{\lambda},\bar{\lambda}+\delta_{2}]$,
\begin{equation}\label{equ123}\begin{split}
v_{0,\lambda}(x)-v(x)&=I_1(x)+I_2(x)\\
&\geq 2m_2-\widetilde{C}\left[\lambda^{n-1}-(\bar{\lambda}-\delta_1)^{n-1}\right] \\
&\geq 2m_2-(n-1)\widetilde{C}\lambda^{n-2}\left[\lambda-(\bar{\lambda}-\delta_1)\right] \\
&\geq 2m_2-(n-1)\widetilde{C}\left(2\bar{\lambda}\right)^{n-2}\left[\delta_1+\delta_2\right] \\
&\geq m_2>0, \qquad \forall \, x\in B^+_{\bar{\lambda}-\delta_1}(0),
\end{split}\end{equation}
thus we have arrived at \eqref{vdexiajie}. As a consequence, one has, for any $\lambda\in[\bar{\lambda},\bar{\lambda}+\delta_{2}]$,
\begin{equation}\label{eq-a6}
  B_{\lambda,u}^{-}\subseteq B^{n-1}_{\lambda}(0)\setminus{B^{n-1}_{\bar{\lambda}-\delta_1}(0)}, \qquad B_{\lambda,v}^{-}\subseteq B^{+}_{\lambda}(0)
\setminus{B^{+}_{\bar{\lambda}-\delta_1}(0)}.
\end{equation}

Therefore, being similar to the proof of Lemma \ref{lemmastart}, one can use the narrow regions $B^{n-1}_{\lambda}(0)\setminus{B^{n-1}_{\bar{\lambda}-\delta_1}(0)}$ and $B^{+}_{\lambda}(0)\setminus{B^{+}_{\bar{\lambda}-\delta_1}(0)}$ instead of $B^{n-1}_{\lambda}(0)\setminus{B^{n-1}_{\lambda^{2}}(0)}$ and $B^{+}_{\lambda}(0)\setminus{B^+_{\lambda^{2}}(0)}$ respectively, and derive estimates \eqref{equ20} and \eqref{equ21} by choosing $\delta_1>0$ and hence $\delta_2\in(0,\delta_{1})$ smaller if necessary. Hence we obtain that $B^{-}_{\lambda,u}=B^{-}_{\lambda,v}=\emptyset$ for all $\lambda\in[\bar{\lambda},\bar{\lambda}+\delta_2]$. That is, for any $\lambda\in[\bar{\lambda},\bar{\lambda}+\delta_2]$,
\begin{equation*}
  u(y)-u_{0,\lambda}(y)\geq0, \qquad \forall \, y\in B^{n-1}_{\lambda}(0)\setminus\{0\}
\end{equation*}
and
\begin{equation*}
  v(x)-v_{0,\lambda}(x)\geq0, \qquad \forall \, x\in B^{+}_{\lambda}(0),
\end{equation*}
which contradicts the definition \eqref{defn} of $\bar{\lambda}$. As a consequence, in the case $0<\bar{\lambda}<+\infty$, we must have
\begin{equation}\label{uhengdeng}
  u_{0,\bar{\lambda}}(y)=u(y), \,\,\,\,\,\, \forall \, y\in B^{n-1}_{\bar{\lambda}}(0)\setminus\{0\}
\end{equation}
and
\begin{equation}\label{vhengdeng}
  v_{0,\bar{\lambda}}(x)=v(x), \,\,\,\,\,\, \forall \, x\in B^{+}_{\bar{\lambda}}(0),
\end{equation}
which imply that \eqref{start11} and \eqref{start12} must hold.

From \eqref{equ11}, \eqref{equ12}, \eqref{start11} and \eqref{start12}, one can derive \eqref{ktheta} immediately. This completes our proof of Lemma \ref{lemmasequali}.
\end{proof}

Now we are ready to give a complete proof of Theorem \ref{theoremfenlei}.

{\bf Proof of Theorem \ref{theoremfenlei}.} We carry out the proof by discussing two different possible cases.

\emph{Case (i).}  $\bar{\lambda}(z)=+\infty$ for all $z\in \partial\mathbb{R}^{n}_+$. Therefore, for all $z\in \partial\mathbb{R}^{n}_+$ and $0<\lambda<+\infty$, we have
$$u_{z,\lambda}(y)\leq u(y), \qquad \forall \,\, y\in B^{n-1}_{\lambda}(z)\setminus\{z\}.$$
Then, by a calculus lemma (Lemma 5.7 in \cite{L}, see also Lemma 11.2 in \cite{LZH}), we conclude that $u\equiv C_{0}>0$ in $\partial\mathbb{R}^{n}_{+}$.

On the other hand, for any $z\in \partial\mathbb{R}^{n}_+$ and $0<\lambda<+\infty$, we also have
$$v_{z,\lambda}(x)\leq v(x),\ \ \ \,\,\, \forall \,\, x\in B^{+}_{\lambda}(z).$$
From a calculus lemma (Lemma 15 in \cite{Ngo2}, see also Lemma 2.2 in \cite{LZ} and Lemma 11.3 in \cite{LZH}), we deduce that $v(x)$ only depends on $x_n$. Thus, it follows from \eqref{intsp} that, for any $x\in\mathbb{R}^{n}_{+}$,
\begin{equation}\label{changshu}\begin{split}
v(x)=v(0,x_n)&=\int_{\partial\mathbb{R}^n_{+}}\frac{C_{0}^{-\theta}}{\left(|y|^2+x_n^2\right)^{\frac{n-\alpha}{2}}}dy,\\
&=\omega_{n-2}C_{0}^{-\theta}x_n^{\alpha-1}\int_0^\infty\frac{r^{n-2}}{\left(r^2+1\right)^{\frac{n-\alpha}{2}}}dr=+\infty,
\end{split}\end{equation}
which is absurd. Thus \emph{Case (i)} is impossible, we only need to consider \emph{Case (ii)}.

\medskip

\emph{Case (ii).} There exists a $\hat{z}\in\partial\mathbb{R}^n_{+}$ such that $\bar{\lambda}(\hat{z})<+\infty$. Then we can deduce from the definition of $\bar{\lambda}(\hat{z})$ that, for any $0<\lambda<\bar{\lambda}(\hat{z})$,
$$u_{\hat{z},\lambda}(y)\leq u(y), \qquad \forall \,\, y\in B^{n-1}_{\lambda}(\hat{z})\setminus\{\hat{z}\},$$
moreover, Lemma \ref{lemmasequali} indicates that $k=\frac{n+\alpha+2\beta}{\alpha+2\beta-n}$, $\theta=\frac{n+\alpha-2}{\alpha-n}$ and
\begin{equation}\label{eq-a7}
  u_{\hat{z},\bar{\lambda}(\hat{z})}(y)=u(y),  \qquad \forall \ y\in\partial\mathbb{R}^n_{+}.
\end{equation}

For any $z\in\partial\mathbb{R}^{n}_{+}$, by the definition of $\bar{\lambda}(z)$, one has, for any $0<\lambda\leq\bar{\lambda}(z)$,
$$u_{z,\lambda}(y)\leq u(y), \qquad \forall \,\, y\in B^{n-1}_{\lambda}(z)\setminus\{z\},$$
that is,
$$u(y)\leq u_{z,\lambda}(y), \qquad  \forall \,\,|y-z|\geq\lambda, \qquad \forall \,\, 0<\lambda\leq\bar{\lambda}(z).$$
It follows immediately that, for any $\lambda\in (0,\bar{\lambda}(z)]$,
\begin{eqnarray}\label{eq-a8}
  && \quad \left[\bar{\lambda}(\hat{z})\right]^{n-\alpha}u(\hat{z})=\liminf_{|y|\rightarrow+\infty}|y|^{n-\alpha}u_{\hat{z},\bar{\lambda}(\hat{z})}(y)
  =\liminf_{|y|\rightarrow+\infty}|y|^{n-\alpha}u(y) \\
 \nonumber &&\leq\liminf_{|y|\rightarrow+\infty}|y|^{n-\alpha}u_{z,\lambda}(y)=\lambda^{n-\alpha}u(z),
\end{eqnarray}
which yields that $\bar{\lambda}(z)<+\infty$ for all $z\in\partial\mathbb{R}^n_{+}$.

Then, from Lemma \ref{lemmasequali}, we infer that, for all $z\in\partial\mathbb{R}^{n}_{+}$,
\begin{equation}\label{4-3}
  u_{z,\bar{\lambda}(z)}(y)=u(y) \quad \text{and} \quad v_{z,\bar{\lambda}(z)}(x)=v(x),   \qquad \forall \ y\in\partial\mathbb{R}^n_{+}, \quad \forall \ x\in\mathbb{R}^n_{+}.
\end{equation}
Since equation \eqref{intsp} is conformally invariant, from the calculus lemma (Lemma 5.8 in \cite{L}, see also Lemma 2.5 in \cite{LZ}, Lemma 11.1 in \cite{LZH} and Theorem 1.4 in \cite{FL1}) and \eqref{4-3}, we deduce that, for any $y\in\partial\mathbb{R}^{n}_+$,
\begin{equation*}
u(y)=c_1\left(\frac{d}{1+d^{2}|y-z_0|^2}\right)^{\frac{n-\alpha}{2}} \qquad \text{and} \qquad v(y,0)=c_2\left(\frac{d}{1+d^{2}|y-z_0|^2}\right)^{\frac{n-\alpha}{2}}
\end{equation*}
for some $c_1>0$, $c_2>0$, $d>0$ and $z_0\in \partial\mathbb{R}^{n}_+$.

\medskip

Now we are to compute the best constant for inequality \eqref{dengjiabudengshi}. To this end, define
$$y^{x_0,\lambda}:=\frac{\lambda^2(y-x_0)}{|y-x_0|^2}+x_0$$
and
$$f_{x_0,\lambda}(y):=\left(\frac{\lambda}{|y-x_0|}\right)^{n+\alpha-2}f(y^{x_0,\lambda}),$$
where $y\in\overline{\mathbb{R}^n_{+}}$, $x_0=(0,-\lambda)\in \mathbb{R}^{n-1}\times \mathbb{R}^{-}$ and $\lambda>0$.
A direct computation gives
$$\|f\|_{L^{\frac{2(n-1)}{n+\alpha-2}}(\partial \mathbb R^n_+)}=\|f_{x_0,\lambda}\|_{L^{\frac{2(n-1)}{n+\alpha-2}}\left(\partial B_{\frac{\lambda}{2}}(x_1)\right)}$$
and
$$\|Tf\|_{L^{\frac{2n}{n-\alpha-2\beta}}(\mathbb R^n_+)}=\|\widehat{T}(f_{x_0,\lambda})\|_{L^{\frac{2n}{n-\alpha-2\beta}}\left(B_{\frac{\lambda}{2}}(x_1)\right)},$$
where $x_1=(0,-\frac{\lambda}{2})\in \mathbb{R}^{n-1}\times \mathbb{R}^{-}$ and
$$\widehat{T}(f_{x_0,\lambda})(\xi):=\int_{\partial B_{\frac{\lambda}{2}}(x_1)}\left(\frac{\lambda}{4}-\frac{1}{\lambda}|\xi-x_1|^{2}\right)^\beta
\frac{ f_{x_0,\lambda}(\eta)}{|\eta-\xi|^{n-\alpha}}d\eta, \qquad \forall \, \xi\in B_{\frac{\lambda}{2}}(x_1).$$
Here we have used the following point-wise identity:
\[\lambda^{-1}|\xi-x_{0}|^{2}\left(\xi^{x_{0},\lambda}\right)_{n}=\frac{\lambda^{2}}{4}-|\xi-x_{1}|^{2}, \qquad \forall \, \xi\in B_{\frac{\lambda}{2}}(x_1).\]

From Theorem \ref{theoremfenlei}, we know that
$$\hat{f}(y):=\left(\frac{\lambda}{|y-x_0|}\right)^{n+\alpha-2}, \qquad \forall \,y\in\partial \mathbb{R}^n_+$$
is an extremal function to inequality \eqref{dengjiabudengshi} for any $\lambda>0$ and $x_0=(0,-\lambda)$. For any $y\in\partial B_{\frac{\lambda}{2}}(x_1)$, we have $\hat{f}_{x_0,\lambda}(y)=1$ and hence
\begin{equation*}\begin{split}
& \quad \|T\hat{f}\|_{L^{\frac{2n}{n-\alpha-2\beta}}(\mathbb R^n_+)}=\|\widehat{T}(\hat{f}_{x_0,\lambda})\|_{L^{\frac{2n}{n-\alpha-2\beta}}\left(B_{\frac{\lambda}{2}}(x_1)\right)}\\
&=\left(\int_{ B_{\frac{\lambda}{2}}(x_1)}\left|\int_{\partial B_{\frac{\lambda}{2}}(x_1)}\left(\frac{\lambda}{4}-\frac{1}{\lambda}|\xi-x_1|^{2}\right)^\beta
\frac{\hat{f}_{x_0,\lambda}(\eta)}{|\eta-\xi|^{n-\alpha}}d\eta\right|^{\frac{2n}{n-\alpha-2\beta}}d\xi\right)^{\frac{n-\alpha-2\beta}{2n}}\\
&=\left(\frac{\lambda}{2}\right)^{\frac{n+\alpha-2}{2}}\left(\int_{ B_1}
\left|\int_{\partial B_1}\left(\frac{1-|\xi|^2}{2}\right)^\beta|\eta-\xi|^{\alpha-n}
d\eta\right|^{\frac{2n}{n-\alpha-2\beta}}d\xi\right)^{\frac{n-\alpha-2\beta}{2n}}.
\end{split}\end{equation*}
On the other hand, one has
$$\|\hat{f}\|_{L^{\frac{2(n-1)}{n+\alpha-2}}(\partial \mathbb R^n_+)}=\|\hat{f}_{x_0,\lambda}\|_{L^{\frac{2(n-1)}{n+\alpha-2}}\left(\partial B_{\frac{\lambda}{2}}(x_1)\right)}=\left(n\nu_n\right)^{\frac{n+\alpha-2}{2(n-1)}}\left(\frac{\lambda}{2}\right)^{\frac{n+\alpha-2}{2}},$$
where $\nu_{n}$ denotes the volume of the unit ball in $\mathbb{R}^{n}$. As a consequence, we have, in the conformally invariant case $p=\frac{2(n-1)}{n+\alpha-2}$ and $q=\frac{2n}{n-\alpha-2\beta}$, the best constant in inequality \eqref{dengjiabudengshi} is
\begin{equation*}\begin{split}
C^*_{n,\alpha,\beta,p}
&=\|\widehat{T}(\hat{f}_{x_0,\lambda})\|_{L^{\frac{2n}{n-\alpha-2\beta}}\left(B_{\frac{\lambda}{2}}(x_1)\right)}\cdot
\|\hat{f}_{x_0,\lambda}\|^{-1}_{L^{\frac{2(n-1)}{n+\alpha-2}}\left(\partial B_{\frac{\lambda}{2}}(x_1)\right)}\\
&=\left(n\nu_n\right)^{-\frac{n+\alpha-2}{2(n-1)}}
\left(\int_{ B_1}\left|\int_{\partial B_1}\left(\frac{1-|\xi|^2}{2}\right)^\beta|\eta-\xi|^{\alpha-n}
d\eta\right|^{\frac{2n}{n-\alpha-2\beta}}d\xi\right)^{\frac{n-\alpha-2\beta}{2n}}.
\end{split}\end{equation*}

This concludes our proof of Theorem \ref{theoremfenlei}.

\section{The proof of Theorem \ref{theorem4}}

In this section, we investigate the necessary condition for the existence of positive solutions to the following integral system:
\begin{equation}\label{integralaa}\begin{cases}
u(y)=\int_{\mathbb{R}^n_+}\frac{x_n^{\beta(1-k)}}{|x-y|^{n-\alpha}}v^{-k}(x)dx, \qquad y\in\partial\mathbb{R}^n_+,\\ \\
v(x)=\int_{\partial\mathbb{R}^n_+}\frac{1}{|x-y|^{n-\alpha}}u^{-\theta}(y) dy, \qquad x\in\mathbb{R}^n_+.
\end{cases}\end{equation}

In order to derive Theorem \ref{theorem4}, it is sufficient for us to prove the following theorem via using Pohozaev type identities.
\begin{thm}\label{theorem4sub}
For $n\geq2$, $\alpha>n$, $\beta\geq0$, $\theta>0$, $k>0$, $\theta\neq1$, $k\neq1$ satisfying $\beta(1-k)+1>0$, assume that the system \eqref{integralaa} admits a pair of positive $C^1$ solutions $(u,v)$, then a necessary condition for $\theta$ and $k$ is
$$\frac{n-1}{\theta-1}+\frac{n}{k-1}=\alpha+\beta-n.$$
\end{thm}
\begin{proof}
Assume that $(u,v)\in C^{1}(\partial\mathbb{R}^n_+)\times C^{1}(\mathbb{R}^n_+)$ is a pair of positive solutions to the integral system \eqref{integralaa}, then it follows from Lemma \ref{lem2} that
\begin{equation}\label{eq-a9}
  0<\int_{\partial \mathbb{R}^n_{+}}u^{1-\theta}(y)dy\leq C_{1}\int_{\partial \mathbb{R}^n_{+}}(1+|y|^{\alpha-n})u^{-\theta}(y)dy<+\infty,
\end{equation}
\begin{equation}\label{eq-a10}
0<\int_{\mathbb{R}^n_{+}}x_n^{\beta(1-k)}v^{1-k}(x)dx\leq C_{2}\int_{\mathbb{R}^n_{+}}x_n^{\beta(1-k)}(1+|x|^{\alpha-n})v^{-k}(x)dx<+\infty,
\end{equation}
that is, $(u,x_{n}^{\beta}v)\in L^{1-\theta}(\partial\mathbb{R}^{n}_{+})\times L^{1-k}(\mathbb{R}^{n}_{+})$.

Through integrating by parts, one can derive
\begin{equation}\label{eq-a11}\begin{split}
&\quad \int_{B_R^{n-1}(0)} u^{-\theta}(y)(y\cdot\nabla u(y))dy\\
&=\frac{1}{1-\theta}\int_{B_R^{n-1}(0)} y\cdot\nabla (u^{1-\theta}(y))dy\\
&=\frac{R}{1-\theta}\int_{\partial B_R^{n-1}(0)}u^{1-\theta}(y)d\sigma-\frac{n-1}{1-\theta}\int_{B_R^{n-1}(0)} u^{1-\theta}(x)dx.
\end{split}\end{equation}
Similarly, through integrating by parts, one can also get
\begin{equation}\label{eq-a12}\begin{split}
&\quad \int_{B_R^+(0)}(x_{n}^{\beta}v)^{-k}(x)\left(x\cdot\nabla(x_{n}^{\beta}v)(x)\right)dx \\
&=\frac{R}{1-k}\int_{\partial B_R(0)\cap\mathbb{R}^{n}_{+}}\left(x_{n}^{\beta}v\right)^{1-k}(x)d\sigma-\frac{n}{1-k}\int_{B_R^+(0)}\left(x_{n}^{\beta}v\right)^{1-k}(x)dx.
\end{split}\end{equation}

It follows from $(u,x_{n}^{\beta}v)\in L^{1-\theta}(\partial\mathbb{R}^n_+)\times L^{1-k}(\mathbb{R}^n_+)$ that, there exists a sequence $\{R_{j}\}$ with $R_j\rightarrow+\infty$ such that
$$R_j\int_{\partial B_{R_j}^{n-1}(0)} u^{1-\theta}(y)d\sigma\rightarrow0, \qquad  R_j\int_{\partial B_{R_j}(0)\cap\mathbb{R}^{n}_{+}}\left(x_{n}^{\beta}v\right)^{1-k}(x)d\sigma\rightarrow0.$$
Therefore, by taking $R=R_{j}$ and letting $j\rightarrow+\infty$ in \eqref{eq-a11} and \eqref{eq-a12}, we get
\begin{equation}\label{poh}\begin{split}
&\quad \int_{\partial\mathbb{R}^n_+} u^{-\theta}(y)\left(y\cdot\nabla u(y)\right)dy+\int_{\mathbb{R}^n_+}\left(x_{n}^{\beta}v\right)^{-k}(x)\left(x\cdot\nabla\left(x_{n}^{\beta}v\right)(x)\right)dx \\
&=-\frac{n-1}{1-\theta}\int_{\partial\mathbb{R}^n_+} u^{1-\theta}(x)dx-\frac{n}{1-k}\int_{\mathbb{R}^n_+}\left(x_{n}^{\beta}v\right)^{1-k}(x)dx.
\end{split}\end{equation}

On the other hand, by using \eqref{integralaa}, one can calculate that
\begin{equation*}\begin{split}
\nabla u(y)\cdot y&=\frac{d[u(\rho y)]}{d\rho}\bigg|_{\rho=1}\\
&=-(n-\alpha)\int_{\mathbb{R}^n_+} \frac{x_n^{\beta(1-k)}}{|x-y|^{n-\alpha+2}}\left[(y-x)\cdot y\right]v^{-k}(x)dx,
\end{split}\end{equation*}
and
\begin{equation*}\begin{split}
\nabla\left(x_{n}^{\beta}v\right)(x)\cdot x&=\frac{d[\left(x_{n}^{\beta}v\right)(\rho x)]}{d\rho}\bigg|_{\rho=1}\\
&=-(n-\alpha)\int_{\partial\mathbb{R}^n_+} \frac{x_n^\beta }{|x-y|^{n-\alpha+2}}[(x-y)\cdot x]u^{-\theta}(y)dy\\
&\ \ \ \ +\beta\int_{\partial\mathbb{R}^n_+} \frac{x_n^\beta }{|x-y|^{n-\alpha}} u^{-\theta}(y)dy.
\end{split}\end{equation*}
Consequently, it follows from Fubini's theorem and \eqref{integralaa} that
\begin{equation}\label{eq-a13}\begin{split}
&\quad \int_{\partial\mathbb{R}^n_+}u^{-\theta}(y)\left(y\cdot\nabla u(y)\right)dy+\int_{\mathbb{R}^n_+}\left(x_{n}^{\beta}v\right)^{-k}(x)\left(x\cdot\nabla\left(x_{n}^{\beta}v\right)(x)\right)dx \\
&=(\alpha+\beta-n)\int_{\mathbb{R}^n_+}\int_{\partial\mathbb{R}^n_+} \frac{x_n^{\beta(1-k)}}{|x-y|^{n-\alpha}} u^{-\theta}(y)v^{-k}(x)dydx\\
&=(\alpha+\beta-n)\int_{\partial\mathbb{R}^n_+}u^{1-\theta}(y)dy\\
&=(\alpha+\beta-n)\int_{\mathbb{R}^n_+}\left(x_{n}^{\beta}v\right)^{1-k}(x)dx.
\end{split}\end{equation}

Combining \eqref{eq-a13} with \eqref{poh} implies that $\frac{n-1}{\theta-1}+\frac{n}{k-1}=\alpha+\beta-n$. This finishes our proof of Theorem \ref{theorem4sub}.
\end{proof}

\section*{Acknowledgements}
This work was completed when the third author was visiting the Department of Mathematical Science, Tsinghua University. He thanks the department for its hospitality.

\end{document}